\documentclass[12pt]{article}

\usepackage[T1]{fontenc}
\usepackage[active]{srcltx}
\usepackage{lmodern}
\usepackage{amsmath,amstext,amssymb,graphicx,graphics,color}
\usepackage{theorem}
\usepackage{cite}               
\usepackage{hyperref}           
\usepackage{dsfont}             

\theorembodyfont{\normalfont\slshape} 
\newtheorem{definition}{Definition}
\newtheorem{proposition}{Proposition}[section]
\newtheorem{theorem}[proposition]{Theorem}

\newtheorem{lemma}[proposition]{Lemma}
\theoremstyle{break} 

\newenvironment{proof}%
{{\par\noindent \bf Proof. \nobreak}}%
{\nobreak \removelastskip \nobreak \hfill $\Box$ \medbreak}
\newenvironment{proof2}%
{{\par\noindent \bf Proof \nobreak}}%
{\nobreak \removelastskip \nobreak \hfill $\Box$ \medbreak}
{{\par\noindent \bf Proof lemma. \nobreak}}%
{\nobreak \removelastskip \nobreak \bf End proof lemma. \medbreak}

\usepackage{fancyhdr}
 \fancypagestyle{plain}{
  \fancyhead{}
  
}

\def\paragraph#1{{\bf #1\ }}

\newcommand{\expo}{\mathrm{e}}

\newcommand{\ds}{\displaystyle}

\makeatletter
\def\iddots{\mathinner{\mkern1mu\raise\p@
    \vbox{\kern7\p@\hbox{.}}\mkern2mu
    \raise4\p@\hbox{.}\mkern2mu\raise7\p@\hbox{.}\mkern1mu}}
\makeatother

\title{A macroscopic model for a system of swarming agents using curvature control}
\author{P. Degond\footnotemark[1] \footnotemark[2] \and S. Motsch\footnotemark[3]}
\date{}

\begin{document}

\renewcommand{\thefootnote}{\fnsymbol{footnote}}

\footnotetext[1]{Universit\'{e} de Toulouse, UPS, INSA, UT1, UTM, Institut de
  Math\'{e}matiques de Toulouse, F-31062 Toulouse, France,
  \url{pierre.degond@math.univ-toulouse.fr}} \footnotetext[2]{CNRS, Institut de
  Math\'{e}matiques de Toulouse UMR 5219, F-31062 Toulouse, France}
\footnotetext[3]{Center of Scientific Computation and Mathematical Modeling (CSCAMM),
  University of Maryland, College Park, MD 20742, USA, \url{smotsch@cscamm.umd.edu}}

\renewcommand{\thefootnote}{\arabic{footnote}}

\maketitle

\vspace{0.5 cm}
\begin{abstract}
  In this paper, we study the macroscopic limit of a new model of collective
  displacement. The model, called PTWA, is a combination of the Vicsek alignment
  model\cite{vicsek_novel_1995} and the Persistent Turning Walker (PTW) model of motion by
  curvature control \cite{gautrais_analyzing_2009,degond_large_2008}. The PTW model was
  designed to fit measured trajectories of individual fish
  \cite{gautrais_analyzing_2009}. The PTWA model (Persistent Turning Walker with
  Alignment) describes the displacements of agents which modify their curvature in order
  to align with their neighbors. The derivation of its macroscopic limit uses the
  non-classical notion of generalized collisional invariant introduced in
  \cite{degond_continuum_2008}. The macroscopic limit of the PTWA model involves two
  physical quantities, the density and the mean velocity of individuals. It is a system of
  hyperbolic type but is non-conservative due to a geometric constraint on the
  velocity. This system has the same form as the macroscopic limit of the Vicsek model
  \cite{degond_continuum_2008} (the 'Vicsek hydrodynamics') but for the expression of the
  model coefficients. The numerical computations show that the numerical values of the
  coefficients are very close.  The 'Vicsek Hydrodynamic model' appears in this way as a
  more generic macroscopic model of swarming behavior as originally anticipated.
\end{abstract}


\vskip 0.4cm


\medskip
\noindent {\bf Acknowledgements: } This work has been supported by the Marie Curie Actions
of the European Commission in the frame of the DEASE project (MEST-CT-2005-021122) and by
the french 'Agence Nationale pour la Recherche (ANR)' in the frame of the contracts
'Panurge' (ANR-07-BLAN-0208-03) and 'Pedigree' (ANR-08-SYSC-015-01). The work of S. Motsch
is partially supported by NSF grants DMS07-07949, DMS10-08397 and FRG07-57227.

\vskip 0.4cm

\section{Introduction}
\label{sec:intro}
\setcounter{equation}{0}

Modeling swarming behavior has attracted a lot of attention in the recent years. To model
a flock of birds \cite{ballerini_interaction_2008}, a school of fish
\cite{couzin_collective_2002, hemelrijk_self-organized_2008, viscido_factors_2007,
  parrish_self-organized_2002} or the displacement of ants \cite{jeanson_modulation_2004,
  couzin_self-organized_2003, theraulaz_spatial_2002}, a key question is to understand how
to relate the collective behavior of large groups of agents to simple individual
mechanisms \cite{camazine_self-organization_2001, couzin_self-organization_2003}. From a
mathematical point of view, this question takes the form of the derivation of macroscopic
equations from individual based models\cite{degond_continuum_2008,
  bertin_boltzmann_2006,bellomo_modeling_2008,chuang_state_2007,filbet_derivation_2005}. This
paper is devoted to the derivation of a macroscopic model for a new type of model of
collective behavior where agents control their motion by changing the curvature of their
trajectory. This model has been shown to provide the best fit of fish trajectories
\cite{gautrais_analyzing_2009}.

Among models of collective displacements, the so-called Vicsek model has received a
particular attention \cite{vicsek_novel_1995,czirok_collective_2000}. This model describes
the tendency of individuals to align with their congeners. Many features of this model
have been studied such as the existence of a critical point
\cite{vicsek_novel_1995,chate_modeling_2008}, the long time
behavior\cite{nagy_new_2007,chate_modeling_2008} or the derivation of a continuum
model\cite{degond_continuum_2008,bertin_boltzmann_2006}. Due to its simplicity, several
extensions or modifications of this model have been proposed, such as the Cucker-Smale
model
\cite{cucker_emergent_2007,ha_simple_2009,ha_particle_2008,carrillo_asymptotic_2010,canizo_well-posedness_2009}. There
is also a variety of models which add an attraction and a repulsion rule to the Vicsek
model \cite{chate_modeling_2008,degond_congestion_2009}. However, the Vicsek model has
been proposed on phenomenological bases. By contrast, the experiments of
\cite{gautrais_analyzing_2009} have shown that the Persistent Turning Walker (PTW) model
provides the best fit to individual fish trajectories. In the PTW model, the individual
controls its motion by acting on the curvature of its trajectory instead of acting on its
velocity. However, in its version of \cite{gautrais_analyzing_2009,degond_large_2008}, the
PTW model only describes the evolution of a single individual. The model does not take
into account the interactions between congeners.

In the present work, interactions between individuals are introduced in the PTW model by
means of an alignment rule, like in the Vicsek model. The resulting model, called PTWA
(Persistent Turning Walker with Alignment) describe how each individual is influenced by
the average velocity of its surrounding neighbors. In the framework of the PTW model where
individuals control their motion by acting of the curvature of their trajectory, this
influence must lead to a modification of this curvature. This contrasts with the Vicsek
model, where particles are directly modifying their velocity as a result of the
interaction.

The PTWA model is based on the assumption that the subjects use the time derivative of
their trajectory curvature (or of their acceleration) as a control variable for planning
their movement. Such models are not commonplace in the literature. Their first occurrence
is, to the best or our knowledge, in \cite{gautrais_analyzing_2009,degond_large_2008}. The
present work is the first one in which interaction among the agents is taken into account
within this kind of models (see also \cite{gautrais__????}). We note that
\cite{szabo_turning_2008} introduces the acceleration of neighbors in the rule updating the
subjects' velocities in a variant of the Vicsek Individual-Based model
\cite{vicsek_novel_1995} but motion planning is eventually made by updating the velocity
and not the acceleration.

Once the PTWA model is set up, the main task of the present paper is to derive the
macroscopic limit of this new model. This macroscopic limit is intended to provide a
simplified description of the system at large scales. The major problem for this
derivation is that there is nothing like momentum or energy conservation in the PTWA
model. Such conservation laws are the corner stone of the classical theory of macroscopic
limits in kinetic theory \cite{degond_macroscopic_2004,cercignani_boltzmann_1988}. Indeed,
as a consequence of this absence of conservation, the dimension of the manifold of local
equilibria in the PTWA model is larger than the dimension of the space of collisional
invariants. Conservation laws are therefore missing for providing a closed set of
equations for the macroscopic evolution of the parameters of the local equilibria.  To
overcome this problem, we use the notion of generalized collisional invariant introduced
in \cite{degond_continuum_2008}. Thanks to this new notion, a closed set of macroscopic
equation for the PTWA model can be derived.

The macroscopic model consists of a conservation equation for the local particle density
and an evolution equation for the average velocity. The latter is constrained to be of
unit norm. The resulting system is a non-conservative hyperbolic which shows similarities
but also striking differences to the Euler system of gas dynamics. It has also the same
form as the previously derived macroscopic limit of the Vicsek model (also referred to as
the 'Vicsek Hydrodynamic model') in \cite{degond_continuum_2008}, but for the expression
of the model coefficients. At the end of the paper, we propose a numerical method to
compute the generalized collisional invariant out of which the coefficients of the
macroscopic model are derived. The similarity between the 'Vicsek hydrodynamics' and the
'PTWA hydrodynamics' can be better understood by considering the relations between the
microscopic models. Indeed, the Vicsek model can be seen as a special limit of the PTWA
model in a well-suited asymptotic limit. Work is in progress to establish this connexion
firmly.

The inclusion of the alignment rule in the PTW model changes drastically the large scale
dynamics of the system. Without this alignment rule, the PTW model exhibits a diffusive
behavior at large scales \cite{cattiaux_asymptotic_2010,degond_large_2008}. By contrast,
when the alignment rule is included, the model becomes of hyperbolic type. Therefore, the
local alignment rule added to the PTW model generates convection at the macroscopic scale.

Since the addition of the alignment rule modifies drastically the dynamics of the PTW
model, it is also interesting to study the large scale effects of other types of local
rules such as attraction-repulsion. The goal is to find a common framework for the large
scale dynamics of a large class of swarming models.  Currently, there exist a profusion of
individual based models, especially for fish behavior (see
\cite{parrish_self-organized_2002} for a short review). In a macroscopic model, only the
gross features of the microscopic model remain. Therefore, the derivation of macroscopic
models may be a tool to better capture the common features and differences between these
different types of swarming models.

The outline of the paper is as follows: in section 2, we introduce the PTWA model and the
main result is stated. Section 3 is devoted to the proof of the derivation of the
macroscopic limit of the PTWA model. In section 4, we study some properties of the
so-obtained macroscopic model and we numerically estimate the involved
coefficients. Finally, in section 5, we draw a conclusion of this work.

\section{Presentation of the model and main result}
\label{sec:presentation}
\setcounter{equation}{0}

\subsection{The individual based model}
\label{subseq:IBM}

The starting point is a model in which alignment interaction between agents is introduced
inside the Persistent Turning Walker model (PTW)
\cite{gautrais_analyzing_2009,degond_large_2008}. The PTW model is a model for individual
displacements which has been derived to fit experimentally observed trajectories of
fish. It supposes that individuals control their motion by acting of the curvature of
their trajectory. To make it a realistic model for collective displacements, the PTW model
must be enriched by introducing inter-individual interactions. Indeed, one of the main
features of collective motion such as those observed in animal populations (fish schools,
mammalian herds, etc.) is the ability of individuals to coordinate with each
other. Observations suggest that trend to alignment is an important component of this
interaction and leads to a powerful coordination-building mechanism by synchronizing the
agent's velocities one to each other. One of the simplest models of alignment interaction
is the Vicsek model \cite{vicsek_novel_1995}. This time-discrete model supposes that
individuals move at constant speed and align to the average velocity of their neighbors
(up to a certain stochastic uncertainty) at each time step. A time-continuous version of
this dynamics has been derived in \cite{degond_continuum_2008}.

In order to combine the PTW displacement model and the Vicsek alignment interaction model
(in the time-continuous framework of \cite{degond_continuum_2008}), we propose the
following model further referred to ad the PTWA model (PTW model with alignment): among a
population of $N$ agents, the motion of the $\text{i}^{\text{th}}$ individual is given by
\begin{eqnarray}
  \label{eq:x}
  \frac{d{\bf x}_i}{dt} &=& c \vec{\tau}(\theta_i), \\
  \label{eq:theta}
  \frac{d\theta_i}{dt} &=& c \kappa_i , \\
  \label{eq:kappa}
  d \kappa_i &=& a(\upsilon\overline{\kappa}_i-\kappa_i)\,dt + b\, dB_t^i,
\end{eqnarray}
with
\begin{equation} 
  \label{eq:kappa_bar}
  \overline{\kappa}_i = \vec{\tau}(\theta_i) \times \overline{\Omega}_i
\end{equation}
and
\begin{equation} 
  \label{eq:Omega_bar}
  \overline{\Omega}_i =  \frac{{\bf J}_i}{|{\bf J}_i|}  \quad , \quad \ds
  {\bf J}_i = c\!\!\!\!\!\!\sum_{|{\bf x}_i-{\bf x}_j|<R}\!\!\! \vec{\tau}(\theta_j),
\end{equation}
where ${\bf x}=(x_1,x_2) \in \mathbb{R}^2$ is the position of the individual,
$\vec{\tau}(\theta_i)=(\cos \theta_i\, , \, \sin \theta_i)$ is the direction of its
velocity vector, with the angle $\theta_i \in (-\pi,\pi]$ measured from the $x_1$
direction, $\kappa_i \in \mathbb{R}$ is the curvature of its trajectory and $B_t^i$ is a
standard Brownian motion (with $B_t^i$ independent of $B_t^j$ for $i\neq j$). The
magnitude of the velocity is constant and denoted by $c>0$. The constant $a$ is a
relaxation frequency and $b$ quantifies the intensity of the random perturbation of the
curvature. The vector $\overline{\Omega}_i$ is the mean direction of the neighbors of the
$\text{i}^{\text{th}}$ individual (defined as the individuals $j$ which are at a distance
less than $R$ from ${\bf x}_i$, $R>0$ being the perception distance of the individuals,
supposed given).

The trend to alignment is modeled by the relaxation term of (\ref{eq:kappa}) (in factor of
$a$). It describes the relaxation of the trajectory curvature to the target curvature
$\overline{\kappa}_i$. $\overline{\kappa}_i$ is computed by taking the cross product\footnote{For
  two-dimensional vectors $\vec a = (a_1,a_2)$, $\vec b = (b_1,b_2)$, the cross product
  $\vec a \times \vec b$ is the scalar $a_1 b_2 - a_2 b_1$.} of the direction of the individual
$\vec{\tau}(\theta_i)$ and the mean direction of its neighbors $\overline{\Omega}_i$. $\upsilon \overline{\kappa}_i$ is
the trajectory curvature the individual must achieve in order to align to its
neighbors. It increases with increasing difference between the individual's velocity and
the target velocity. $\upsilon$ is the typical value of the individuals' trajectory curvature and
can be seen as the 'comfort' curvature. The larger $\upsilon$ is, the faster alignment
occurs. The second term of (\ref{eq:kappa}) (in factor of $b$) is a random term which
describes the tendency of individuals to desynchronize to their neighbors in order for
instance, to explore their environment. At equilibrium, these two antagonist effects lead
to a stationary distribution of curvatures which is the building block of the construction
of the macroscopic model.


We illustrate this model in figure \ref{fig:fish_double}. In the left figure, a fish is
represented turning to the left. However, its neighbors are moving towards the other
direction ($\overline{\Omega}$ is pointing to the right). Then the fish is going to adjust
its curvature in order to move towards the same direction as $\overline{\Omega}$ (right
figure). The adjustment of its curvature requires a certain time of order $1/a$, after
which the curvature $\kappa$ is close to $\upsilon \overline{\kappa}$. Therefore, in this
model, there is a time delay between the current acceleration of the fish ($\kappa$) and
its desired acceleration ($\upsilon \overline{\kappa}$). In most models describing animal
behavior, the dynamics is inspired by Newton's second law: the acceleration of an
individual is equal to a force term which incorporates all information about the
environment. In the present model, individuals need a certain time to adjust their
acceleration. This rule can be seen as a modification of Newton's second law saying that
the force is proportional to the time derivative of the acceleration rather than to the
acceleration itself.

\begin{figure}[ht]
  \centering
  \includegraphics[scale=.7]{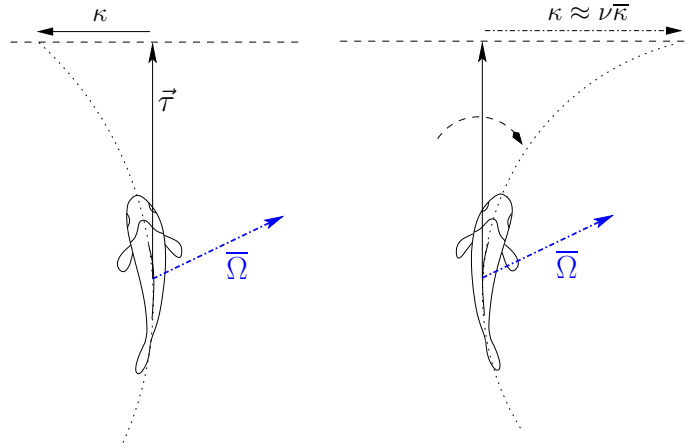}
  \caption{Illustration of the model (\ref{eq:x})-(\ref{eq:kappa}). On the left figure, a
    fish is turning to the left, while its neighbors are moving to the right
    ($\overline{\Omega}$). After a certain time of order $1/a$, the fish adjusts its
    curvature in order to align its velocity with $\overline{\Omega}$ (right figure).}
  \label{fig:fish_double}
\end{figure}


Our goal is the study of model (\ref{eq:x})-(\ref{eq:kappa}) at large time and space
scales. For this purpose, it is convenient to introduce scaled variables. We use $x_0 =
\upsilon^{-1}$ as space unit, $t_0 = (c\upsilon)^{-1}$ as time unit, $\kappa_0 = x_0^{-1}
= \upsilon$ as curvature unit. We introduce the dimensionless time, space and curvature as
$t'= t/t_0$, ${\bf x}' = {\bf x}/x_0$ and $\kappa' = \kappa/\kappa_0$ and for simplicity
we omit the primes in the discussion below. In scaled variables, the PTWA model is given
by (for the $\text{i}^{\text{th}}$ individual) :
\begin{eqnarray}
  \label{eq:position}
  \frac{d{\bf x}_i}{dt} &=& \vec{\tau}(\theta_i), \\
  \label{eq:angle}
  \frac{d\theta_i}{dt} &=& \kappa_i , \\
  \label{eq:curvature}
  d \kappa_i &=& \lambda(\overline{\kappa}_i-\kappa_i)\,dt + \sqrt{2} \alpha\, dB_t^i, 
\end{eqnarray}
with $\overline{\kappa}_i$ defined by equation (\ref{eq:kappa_bar}),(\ref{eq:Omega_bar})
($c$ being replaced by $1$) and $\lambda$, $\alpha$ given by:
\begin{displaymath}
  \lambda = \frac{a}{c\upsilon} \qquad, \qquad \alpha^2 = \frac{b^2}{2c\upsilon^3}.
\end{displaymath}

\subsection{Main result}
\label{subseq:Main}

A first step consists in providing a mean-field description of the PTWA
dynamics. Introducing the probability density function of fish $f(t,{\bf
  x},\theta,\kappa)$, we will prove formally that the PTWA model
(\ref{eq:position})-(\ref{eq:curvature}) leads to the following equation for $f$:
\begin{equation}
  \label{eq:f}
  \partial_tf + \vec{\tau}(\theta)\cdot\nabla_{{\bf x}}f + \kappa\partial_\theta f +
  \lambda\partial_\kappa\big[(\overline{\kappa}-\kappa)f\big] =
  \alpha^2 \partial_{\kappa}^2f,
\end{equation}
with
\begin{equation}
  \label{eq:kappa_bar_kinetic}
  \overline{\kappa} = \vec{\tau}(\theta)\times\overline{\Omega}({\bf x})
\end{equation}
and
\begin{equation}
  \label{eq:theta_bar_kinetic}
  \overline{\Omega}({\bf x}) =  \frac{{\bf J}({\bf x})}{|{\bf J}({\bf x})|} \quad, \quad \ds
  {\bf J}({\bf x}) = \int_{|{\bf x}-{\bf y}|<R,\,\theta,\kappa} \vec{\tau}(\theta)f({\bf
    y},\theta,\kappa)\,d{\bf y} d\theta d\kappa.
\end{equation}


\noindent The main concern of this paper is the study of the so-called hydrodynamic limit
of the mean-field model (\ref{eq:f}). With this aim, we perform a new rescaling and
introduce the macroscopic variables $\tilde t$ and $\tilde {\bf x}$:
\begin{equation}
  \label{eq:macro_variables}
  \tilde t = \varepsilon t \qquad , \qquad \tilde {\bf x} = \varepsilon {\bf x},
\end{equation}
with $\varepsilon>0$ a small number representing the ratio between the microscopic and the
macroscopic time and space scales. In this paper, we give a formal proof that the density
distribution of individuals in these new variables $f^\varepsilon(\tilde t, \tilde {\bf
  x},\theta,\kappa)$ converges in the limit $\varepsilon\rightarrow0$ to the solutions of
a hydrodynamic like model. More precisely, the theorem reads (dropping the tildes for
simplicity):

\begin{theorem}
  \label{thm:formal_cvgnce}
  In the limit $\varepsilon\rightarrow0$, the distribution $f^\varepsilon$ converges to an equilibrium:
  \begin{displaymath}
    f^\varepsilon \stackrel{\varepsilon \rightarrow 0}{\rightharpoonup} \rho \mathcal{M}_{\Omega}(\theta) \mathcal{N}(\kappa)
  \end{displaymath}
  with $\mathcal{M}_{\Omega}$ and $\mathcal{N}$ (resp.) a Von Mises distribution and a
  Gaussian distribution defined at (\ref{eq:Von}) and (\ref{eq:Gauss}). The density
  $\rho=\rho({\bf x},t)$ and the direction of the flux $\Omega=\Omega({\bf x},t)$ satisfy
  the following system:
  \begin{equation}
    \label{eq:final_system}
    \begin{array}{l}
      \partial_t \rho + c_1 \nabla_{{\bf x}} \cdot (\rho\Omega) = 0,\\
      \rho \big(\partial_t \Omega + c_2 (\Omega\cdot\nabla_{{\bf x}}) \Omega\big) +
      \frac{\alpha^2}{\lambda^2}(\mbox{Id}-\Omega\otimes\Omega) \nabla_{\!{\bf x}}\rho =
      0,
    \end{array}
  \end{equation}
  where $c_1$ and $c_2$ are two positive constants defined later on at
  (\ref{eq:c_1})~(\ref{eq:c_2}).
\end{theorem}

The so-obtained macroscopic model (\ref{eq:final_system}) has the same form as that
derived from the Vicsek model \cite{degond_continuum_2008}. Indeed, the two models only
differ by the values of their coefficients. This model is a hyperbolic system which bears
some similarities with the Euler system of isothermal compressible gases. There are
however some striking differences. First, the convection speed of the density $\rho$ is
different from the convection speed of the velocity $\Omega$ ($c_1 \neq c_2$ in
general). Moreover, the velocity $\Omega$ is a unit vector and therefore it satisfies the
constraint $|\Omega|=1$. This explains why the pressure term is premultiplied by the
matrix $(\mbox{Id}-\Omega\otimes\Omega)$. This projection matrix guarantees that the
resulting vector is orthogonal to $\Omega$. Consequently, the constraint $|\Omega|=1$ is
preserved by the dynamics. However, the projection matrix leads to a non-conservative
model which cannot be put in conservative form. This intrinsic non-conservation feature is
the macroscopic counterpart of the lack of momentum conservation at the microscopic level
(see below).

The modification of the PTW model leading to the PTWA model has drastically changed the
nature of the macroscopic model. Indeed, the macroscopic limit of the PTW model without
the incorporation of the interactions is of diffusive nature
\cite{degond_large_2008,cattiaux_asymptotic_2010}. By contrast, that of the PTWA model is
of hyperbolic type. Indeed, the scaling (\ref{eq:macro_variables}) is of hydrodynamic
type, the macroscopic time and space scales being of the same order of magnitude. By
contrast, a diffusive scaling would have required $\tilde t = \varepsilon^2 t$ instead
(see \cite{degond_large_2008,cattiaux_asymptotic_2010}).

The similarity with the 'Vicsek Hydrodynamics' also confirms that that the chosen
interaction rule generates alignment since the PTWA model has the same macroscopic limit
as the Vicsek model. At the microscopic scale, the PTWA and Vicsek models look rather
different, whereas, at the macroscopic scale, they are similar. This is an example of how
the derivation of macroscopic model can be used as a tool to reduce and unify different
types of swarming models in classes leading to similar macroscopic models.


\section{Derivation of a macroscopic model}
\label{sec:derivation}
\setcounter{equation}{0}

\subsection{Mean field equation}
\label{subseq:mean_field}

In this section, we briefly summarize the first step of the derivation of the macroscopic
model, namely the derivation of the intermediate mean-field equation (\ref{eq:f}) from the
particle dynamics (\ref{eq:position})-(\ref{eq:curvature}). In order to derive this mean
field equation, we start by looking at the system without the white noise $dB_i^t$ for a
large (but fixed) number of individuals $N$. In this case, the system reduces to a coupled
system of ordinary differential equations. We denote by
$\{X_i(t),\Theta_i(t),K_i(t)\}_{i=1\dots N}$ the solution of this system on a given time
interval. Following the standard methodology (see e.g. the text book
\cite{spohn_large_1991}), we introduce the so-called empirical distribution $f^N$ given
by:
\begin{equation}
  \label{eq:f_N}
  f^N = \frac{1}{N} \sum_{i=1}^N \delta_{X_i(t)} \otimes \delta_{\Theta_i(t)} \otimes \delta_{K_i(t)}.
\end{equation}
We can easily check that this density distribution satisfies the following equation
(weakly):
\begin{displaymath}
  \partial_tf^N + \vec{\tau}(\theta)\cdot\nabla_{{\bf x}}f^N + \kappa\partial_\theta f^N
  + \lambda\partial_\kappa\big[(\overline{\kappa}^N-\kappa)f^N\big] = 0,
\end{displaymath}
with
\begin{displaymath}
  \overline{\kappa}^N = \vec{\tau}(\theta)\times\overline{\Omega}^N({\bf x})
\end{displaymath}
and
\begin{displaymath}
  \overline{\Omega}^N({\bf x}) =  \frac{{\bf J}^N({\bf x})}{|{\bf J}^N({\bf x})|} \quad, \quad \ds
  {\bf J}^N({\bf x}) = \!\!\!\!\!\sum_{j, \, |{\bf x}-{\bf x}_j|<R}\!\!\!
  \vec{\tau}(\theta_j).
\end{displaymath}
The term $J^N$ can be expressed using the empirical distribution $f^N$:
\begin{displaymath}
  {\bf J}^N({\bf x}) = N\, \int_{|{\bf x}-{\bf y}|<R,\,\theta,\kappa}
  \vec{\tau}(\theta)f^N({\bf y},\theta,\kappa)\,d{\bf y} d\theta d\kappa.
\end{displaymath}
Then it is clear that the formal limit $N\rightarrow\infty$ of $f^N$ satisfies the
following equation:
\begin{displaymath}
  \partial_tf + \vec{\tau}(\theta)\cdot\nabla_{{\bf x}}f + \kappa\partial_\theta f +
  \lambda\partial_\kappa\big[(\overline{\kappa}-\kappa)f\big] = 0,
\end{displaymath}
with $\overline{\kappa}$ given by
(\ref{eq:kappa_bar_kinetic}),(\ref{eq:theta_bar_kinetic}).

When the white noise is added, the situation is more complicated. At the particle level
(\ref{eq:position})-(\ref{eq:curvature}), the system becomes a coupled system of
stochastic differential equations. This implies that the empirical distribution $f^N$
given by (\ref{eq:f_N}) becomes a stochastic measure. In this case, formal considerations
suggest that, in the limit $N\rightarrow\infty$, the distribution function $f$ satisfies
the following Fokker-Planck equation (\ref{eq:f}) with $\overline{\kappa}$ given by
(\ref{eq:kappa_bar_kinetic}),(\ref{eq:theta_bar_kinetic}). For related questions, we refer
the reader to
\cite{sznitman_topics_1989,canizo_well-posedness_2009,bolley_stochastic_2010}.

\subsection{Hydrodynamic scaling}
\label{subseq:hydro}

In order to derive a macroscopic equation from the mean-field equation
(\ref{eq:f})-(\ref{eq:theta_bar_kinetic}), we use the hydrodynamic scaling. With this aim,
we introduce the macroscopic variables $\tilde t$ and $\tilde {\bf x}$ defined by
(\ref{eq:macro_variables}).  In the rescaled variables, the distribution function (denoted
by $f^\varepsilon$) is given by $f^\varepsilon(\tilde t,\tilde {\bf x},\theta,\kappa) =
\frac{1}{\varepsilon^2}\,f( t,{\bf x},\theta,\kappa)$. After omitting the tildes, it
satisfies the following equation:
\begin{equation}
  \label{eq:f_ep}
  \varepsilon(\partial_tf^\varepsilon + \vec{\tau}(\theta)\cdot\nabla_{{\bf
      x}}f^\varepsilon) + \kappa\partial_\theta f^\varepsilon +
  \lambda\partial_\kappa\big[(\overline{\kappa}^\varepsilon-\!\kappa)f^\varepsilon\big] =
  \alpha^2 \partial_\kappa^2f^\varepsilon,
\end{equation}
with
\begin{displaymath}
  \overline{\kappa}^\varepsilon = \vec{\tau}(\theta)\times\overline{\Omega}^\varepsilon({\bf x})
\end{displaymath}
and
\begin{equation}
  \label{eq:omega_kinetic}
  \overline{\Omega}^\varepsilon({\bf x}) =  \frac{{\bf J}^\varepsilon({\bf x})}{|{\bf
      J}^\varepsilon({\bf x})|} \quad, \quad \ds
  {\bf J}^\varepsilon({\bf x}) = \int_{|{\bf x}-{\bf y}|<\varepsilon R,\,\theta,\kappa}
  \vec{\tau}(\theta)f^\varepsilon({\bf y},\theta,\kappa)\,d{\bf y} d\theta d\kappa.
\end{equation}
We note that the expression (\ref{eq:omega_kinetic}) of ${\bf J}^\varepsilon$ supposes
that the radius of interaction between the individuals is tied to the microscopic
scale. This assumption translates the fact that in most biological system, each individual
has only access to information about its close neighborhood. Thanks to this assumption, we
can replace the expression of $\overline{\Omega}^\varepsilon$ by a local expression. This
is precisely stated in the following lemma, the proof of which is obvious and omitted.

\begin{lemma}
  We have the expansion:
  \begin{displaymath}
    \overline{\Omega}^\varepsilon = \Omega_{f^\varepsilon} + O(\varepsilon^2),
  \end{displaymath}
  where
  \begin{displaymath}
    \Omega_{f^\varepsilon}({\bf x}) = \frac{{\bf j}^\varepsilon({\bf x})}{|{\bf
        j}^\varepsilon({\bf x})|}\; \quad \text{and} \quad {\bf j}^\varepsilon({\bf x})=
    \int_{\theta,\kappa} \vec{\tau}(\theta) f^\varepsilon ({\bf x},\theta,\kappa)\,
    d\theta d\kappa.
  \end{displaymath}
\end{lemma}

\medskip
\noindent
Finally, we can simplify (\ref{eq:f_ep}) using the equality:
\begin{displaymath}
  \vec{\tau}(\theta)\times\Omega = \sin(\overline{\theta}-\theta)
\end{displaymath}
with $\overline{\theta}$ such that:
\begin{displaymath}
  \vec{\tau}(\overline{\theta}) = \Omega_{f^\varepsilon}.
\end{displaymath}
With these notations, equation (\ref{eq:f_ep}) can be written as:
\begin{equation}
  \label{eq:f_ep2}
  \varepsilon\big(\partial_tf^\varepsilon + \vec{\tau}(\theta)\cdot\nabla_{{\bf
      x}}f^\varepsilon\big) = Q(f^\varepsilon) + O(\varepsilon^2)
\end{equation}
with the operator $Q$ (below referred to as the 'collision operator') defined by:
\begin{equation}
  \label{eq:Q}
  Q(f) = -\kappa\partial_\theta f - \lambda\sin(\overline{\theta}-\theta)\,\partial_\kappa
  f + \lambda\partial_\kappa (\kappa f) + \alpha^2 \partial_\kappa^2f,
\end{equation}
where $\vec{\tau}(\overline{\theta}) = \Omega_f({\bf x})$ defined as:
\begin{equation}
  \label{eq:Omega}
  \Omega_f({\bf x}) = \frac{{\bf j}({\bf x})}{|{\bf j}({\bf x})|} \quad , \quad {\bf
    j}({\bf x})=\int_{\theta,\kappa} \vec{\tau}(\theta) f({\bf x},\theta,\kappa)\,d\theta
  d\kappa.
\end{equation}
In the sequel, we will drop the $O(\varepsilon^2)$ remainder which has no influence in the
final result.

\subsection{Study of the collision operator}
\label{subseq:collision_operator}

\subsubsection{Equilibria}
\label{subsubsec:equilibria}

In order to study the limit $\varepsilon\rightarrow0$ of the solution $f^\varepsilon$ of
(\ref{eq:f_ep2}), we first have to determine the equilibria of the operator $Q$ defined by
(\ref{eq:Q}). With this aim, we notice that $Q$ can be decomposed as a sum of a formally
skew-adjoint operator and of a formally self-adjoint operator. For the skew-adjoint part,
we introduce the function:
\begin{displaymath}
  H(\theta,\kappa) = -\lambda \cos \theta + \frac{\kappa^2}{2}
\end{displaymath}
and we adopt the convention that for any function $h(\theta,\kappa)$:
\begin{equation}
  \label{eq:convention}
  h_\Omega(\theta,\kappa) = h_{\overline{\theta}}(\theta,\kappa) = h(\theta-\overline{\theta},\kappa),
\end{equation}
with $\vec{\tau}(\overline{\theta})=\Omega$. Using these notations, for any smooth
function $f$, the skew-adjoint part of $Q$ can be written as:
\begin{equation}
  \label{eq:anti_ajdoint}
  -\kappa\partial_\theta f - \lambda\sin(\overline{\theta}-\theta)\,\partial_\kappa f
  = \partial_\theta H_{\overline{\theta}} \, \, \partial_\kappa f - \partial_\kappa
  H_{\overline{\theta}} \, \, \partial_\theta f =\{H_{\overline{\theta}},\,f\}_{(\theta,\kappa)},
\end{equation}
using the Poisson Bracket formalism $\{ \cdot , \cdot \}_{(\theta,\kappa)}$ in the
$(\theta,\kappa)$ space. Therefore, any function of the form $g(H_{\overline{\theta}})$
satisfies $\{H_{\overline{\theta}},\,g(H_{\overline{\theta}})\}=0$. On the other hand, the
self-adjoint part of $Q$ satisfies:
\begin{equation}
  \label{eq:self_adjoint}
  \lambda\partial_\kappa (\kappa f) + \alpha^2 \partial_\kappa^2f =
  \alpha^2 \partial_\kappa \left(\mathcal{N} \partial_\kappa \left(\frac{f}{\mathcal{N}}
    \right) \right),
\end{equation}
with $\mathcal{N}$ the Gaussian distribution with zero mean and variance $\alpha^2/\lambda$:
\begin{equation}
  \label{eq:Gauss}
  \mathcal{N}(\kappa) = \sqrt{\frac{\lambda}{2\pi \alpha^2}}\,\exp \left(-\frac{\lambda
      \kappa^2}{2 \alpha^2} \right).
\end{equation}
In particular, the Gaussian $\mathcal{N}$ is in the kernel of the self-adjoint part of
$Q$. We combine our two previous observations to define the function:
\begin{equation}
  \label{eq:mu_temp}
  \mu(\theta,\kappa) = C \exp \left(-\frac{\lambda}{\alpha^2}H \right) = C \exp
  \left(-\frac{\lambda}{\alpha^2} \left(  \frac{\kappa^2}{2} - \lambda \cos \theta \right)
  \right),
\end{equation}
where $C$ is the normalization constant such that $\int_{(\theta, \kappa)} \mu(\theta,\kappa) \, d \theta \, d \kappa =
1$. This normalization constant is explicitly given below. The translates
$\mu_{\overline{\theta}}$ of $\mu$ in the sense of definition (\ref{eq:convention}) are of the form
$g(H_{\overline{\theta}})$ and are Gaussian distributions in $\kappa$ with variance $\alpha / \sqrt \lambda$. It
follows from a simple computation that $\mu_{\overline{\theta}}$ is an equilibrium for $Q$
(i.e. $Q(\mu_{\overline{\theta}})=0$), for all real values of $\overline{\theta}$.

To simplify the analysis, we introduce the Von Mises distribution $\mathcal{M}$:
\begin{equation}
  \label{eq:Von}
  \ds \mathcal{M}(\theta) = C_0\,\exp\left(\frac{\lambda^2}{\alpha^2} \cos \theta\right),
\end{equation}
where $C_0 = (2\pi I_0(\frac{\lambda^2}{\alpha^2}))^{-1}$ is the normalization constant
such that $\int_{\theta} \mathcal{M}(\theta) \, d \theta = 1$ (with $I_0$ the modified
Bessel function of order $0$).  Therefore, $\mu$ can be written as the product of
$\mathcal{M}$ given by (\ref{eq:Von}) and $\mathcal{N}$ given by (\ref{eq:Gauss}):
\begin{equation}
  \label{eq:mu}
  \mu(\theta,\kappa) = \mathcal{M}(\theta) \mathcal{N}(\kappa),
\end{equation}
and the normalization constant $C$ is given by $C=C_0 \sqrt{\lambda/(2\pi\alpha^2)}$. We
summarize our analysis of $Q$ in the following proposition.

\begin{proposition}
  \begin{itemize}
  \item[i)] The operator $Q$ satisfies:
    \begin{equation}
      \label{eq:entropy}
      \int_{\theta,\kappa} Q(f)\,\frac{f}{\mu_{\overline{\theta}}}\,d\theta d\kappa = - \alpha^2 \int_{\theta,\kappa}
      \frac{\mathcal{N}}{\mathcal{M}_{\overline{\theta}}} \left| \partial_\kappa \left( \frac{f}{\mathcal{N}} 
        \right)\right|^2\,d\theta d\kappa \leq 0,
    \end{equation}
    with $\mu$ defined by (\ref{eq:mu}) and $\overline{\theta}$ such that
    $\vec{\tau}(\overline{\theta})=\Omega_f$ with $\Omega_f$ defined in (\ref{eq:Omega}).
  \item[ii)] The equilibria of $Q$ (i.e. the functions $f(\theta, \kappa) \geq 0$ such
    that $Q(f)=0$) form a two-dimensional manifold $\mathcal{E}$ given by:
    \begin{equation}
      \label{eq:equil}
      \mathcal{E} = \{\rho\, \mu_{\overline{\theta}} \quad | \quad \rho \in
      \mathbb{R}^+\;, \overline{\theta} \in (-\pi,\pi]\},
    \end{equation}
    where $\rho$ is the total mass and $\overline{\theta}$ the direction of the flux of
    $\rho\, \mu_{\overline{\theta}}$.
  \end{itemize}
\end{proposition}

\begin{proof} (i) Combining (\ref{eq:anti_ajdoint}) and (\ref{eq:self_adjoint}), we find:
  \begin{equation}
    \label{eq:Q_simplify}
    Q(f) = \{H_{\overline{\theta}},\,f \} + \alpha^2 \partial_\kappa
    \left(\mathcal{N} \partial_\kappa \left(\frac{f}{\mathcal{N}}\right) \right).
  \end{equation}
  Using (\ref{eq:mu_temp}), the fact that the Poisson bracket with $f$ is a derivation and
  is a skew-adjoint operator, we find:
  \begin{eqnarray*}
    \int_{\theta,\kappa} \{H_{\overline{\theta}},f\}\,\frac{f}{\mu}\,d\theta d\kappa &=&
    \int_{\theta,\kappa} \frac{\alpha^2}{\lambda} \expo^{-\frac{\lambda}{\alpha^2}H_{\overline{\theta}}}
    \{\expo^{\frac{\lambda}{\alpha^2}H_{\overline{\theta}}},f\} \frac{f}{\mu}\,d\theta d\kappa
    \\
    &=& \frac{\alpha^2}{\lambda} \frac{1}{C}\,  \int_{\theta,\kappa}
    \{\expo^{\frac{\lambda}{\alpha^2}H_{\overline{\theta}}},f\} f\,d\theta d\kappa =0.
  \end{eqnarray*}
  Then, using the formulation of $Q$ in (\ref{eq:Q_simplify}), we easily deduce the
  equality (\ref{eq:entropy}) by applying Green's formula.
  
  \medskip
  \noindent
  (ii) If $f$ is an equilibrium for $Q$ (i.e. $Q(f)=0$) using the equality
  (\ref{eq:entropy}) we have:
  \begin{displaymath}
    \int_{\theta,\kappa} \frac{\mathcal{N}(\kappa)}{\mathcal{M}_{\overline{\theta}} (\theta)}
    \left| \partial_\kappa \left( \frac{f}{\mathcal{N}} \right)\right|^2\,d\theta d\kappa
    = 0,
  \end{displaymath}
  which means that $f$ is proportional to $\mathcal{N}$ as a function of
  $\kappa$. Therefore, we can write:
  \begin{displaymath}
    f(\theta,\kappa) = \varphi(\theta)\mathcal{N}(\kappa).
  \end{displaymath}
  Using again that $f$ is an equilibrium, we have:
  \begin{displaymath}
    -\kappa\varphi'(\theta) + \lambda\sin(\overline{\theta}-\theta)
    \frac{\lambda\kappa}{\alpha^2} \varphi(\theta) = 0, \quad \text{for all } \kappa.
  \end{displaymath}
  Solving this differential equation leads to $\varphi(\theta) = C \mathcal{M}_{\overline{\theta}}(\theta)$
  with $\mathcal{M}$ given by (\ref{eq:Von}). This yields $f= K
  \mathcal{M}_{\overline{\theta}} \, \mathcal{N}$ with $K\geq 0$ a constant which proves that
  $f$ is of the form $f = \rho \mu_{\theta_0}$, with $\rho \geq 0$ and $\theta_0 \in
  (-\pi, \pi]$.
  
  \medskip
  \noindent
  Reciprocally, we show that a function of the form $f=\rho \mu_{\theta_0}$ with $\rho
  \geq 0$ and $\theta_0 \in (-\pi, \pi]$ is an equilibrium. For this purpose, the only
  thing to show is that the associated $\Omega_f = \tau (\overline{\theta})$ is such that
  $\overline{\theta}=\theta_0$. We compute
  \begin{eqnarray*} 
    {\bf j}_f &=& \int_{(\theta,\kappa)}  \rho \, \mu_{\theta_0} \, \tau(\theta) \, d\theta \, d \kappa \\
    &=& \rho \int_{(\theta,\kappa)} {\mathcal N}(\kappa) \,  C_0 \, \exp \left(
      \frac{\lambda^2}{\alpha^2} \cos (\theta - \theta_0) \right) \,
    \left( \begin{array}{c} \cos \theta \\ \sin \theta \end{array} \right) \, d\theta \, d\kappa.
  \end{eqnarray*} 
  Then, by the change of variables $\phi = \theta - \theta_0$ and using oddness
  considerations, we obtain
  \begin{displaymath}
    {\bf j}_f = \rho C_0 \, \int_{\kappa} \exp \left( \frac{\lambda^2}{\alpha^2} \cos \phi
    \right) \, \cos \phi \, d\phi  \, \, \tau(\theta_0)= \rho
    \frac{I_1(\frac{\lambda^2}{\alpha^2})}{I_0(\frac{\lambda^2}{\alpha^2})} \,
    \tau(\theta_0),
  \end{displaymath}
  where $I_1$ is the modified Bessel function of order 1. Remembering that $\Omega_f =
  {\bf j}_f / |{\bf j}_f|$, we deduce that $\overline{\theta} = \pm \theta_0$, with the sign
  being that of ${I_1(\frac{\lambda^2}{\alpha^2})}/{I_0(\frac{\lambda^2}{\alpha^2})}$. A
  simple inspection of the integral giving $I_1$ shows that this sign is positive and that
  $\overline{\theta} = \theta_0$, which ends the proof.
\end{proof}

\subsubsection{Generalized collisional invariant}
\label{subsubseq:collision_invar}

The next step to determine the hydrodynamic limit of $f^\varepsilon$ (\ref{eq:f_ep2}) is
to look at the collision invariants of the operator $Q$, i.e. the functions $\psi$ which
satisfy:
\begin{displaymath}
  \int_{\theta,\kappa} Q(f) \psi\,d\theta d\kappa = 0, \qquad \text{ for all } f.
\end{displaymath}
Clearly, $\psi=1$ is collisional invariant. But there is no other obvious collisional
invariant. However, since the equilibria of $Q$ (\ref{eq:equil}) form a two dimensional
space, we need two conserved quantities to derive a macroscopic model. To overcome this
problem, we use the notion of {\it generalized collisional invariant} developed in
\cite{degond_continuum_2008}.

In this paper, we use a slightly different definition from
\cite{degond_continuum_2008}. Indeed, the result of \cite{degond_continuum_2008} was
slightly incorrect and the present definition is designed to make the statement
correct. We first introduce the following definition:

\begin{definition}
  For a given $\Omega \in {\mathbb S}^1$ and a given distribution function $f(\theta,
  \kappa)$, we define the 'extended' collision operator ${\mathcal Q}_\Omega(f)$ by:
  \begin{displaymath}
    {\mathcal Q}_\Omega(f) = \{H_{\Omega},\,f \} + \alpha^2 \partial_\kappa
    \left(\mathcal{N} \partial_\kappa \left(\frac{f}{\mathcal{N}}\right) \right),
  \end{displaymath}
  where we recall the notation (\ref{eq:convention}).
  \label{def:ext_coll}
\end{definition}

\noindent
Obviously, we have
\begin{equation}
  Q(f) = {\mathcal Q}_{\Omega_f}(f) ,
  \label{eq:rel_coll_ext_coll}
\end{equation}
recalling the definition (\ref{eq:Omega}) of $\Omega_f$. For fixed $\Omega$, the operator
${\mathcal Q}_\Omega(f)$ is linear. We now define a Generalized Collision Invariant.

\begin{definition}
  For a given unit vector $\Omega \in \mathbb{S}^1$, a function $\psi_\Omega$ is called a
  Generalized Collisional Invariant (GCI) if it satisfies:
  \begin{equation}
    \label{eq:generalized_collisional_invariant}
    \int_{\theta,\kappa} {\mathcal Q}_\Omega(f) \,  \psi_\Omega \, d\theta \, d\kappa = 0,
    \quad \text{ for all } f \text{ such that } \Omega_f= \pm \Omega,
  \end{equation}
\end{definition}

\noindent
Using definition (\ref{eq:generalized_collisional_invariant}) with $\Omega_f = \Omega$ and
(\ref{eq:rel_coll_ext_coll}), we note that if $\psi_\Omega$ is a GCI, it satisfies
\begin{displaymath}
  \int_{\theta,\kappa} Q(f) \,  \psi_{\Omega_f} \, d\theta \, d\kappa = 0.
\end{displaymath}
This property is crucial for the establishment of the hydrodynamic limit. 

\medskip For a given $\Omega \in {\mathbb S}^1$, the adjoint operator to ${\mathcal
  Q}_\Omega$ is given by:
\begin{displaymath}
  {\mathcal Q}_\Omega^* (\psi) = \kappa\partial_\theta\psi +
  \lambda\sin(\overline{\theta}-\theta)\,\partial_\kappa\psi -
  \lambda\kappa\partial_\kappa\psi + \alpha^2 \partial_\kappa^2\psi,
\end{displaymath}
with $\overline{\theta}$ such that $\Omega = \vec{\tau} (\overline{\theta})$. This operator ${\mathcal
  Q}_\Omega^*$ enables us to find an explicit equation for the GCI $\psi_\Omega$ as stated
in the following lemma.

\begin{lemma}
  For a given unit vector $\Omega \in \mathbb{S}^1$, a function $\psi_\Omega$ is a
  generalized collisional invariant if and only if it there exists a constant $\beta \in
  {\mathbb R}$ such that:
  \begin{equation}
    \label{eq:Q_adjoint_sin}
    {\mathcal Q}_\Omega^* (\psi_\Omega) = \beta\,\vec{\tau}(\theta) \times \Omega.
  \end{equation}
\end{lemma}

\begin{proof}
  Let $f(\theta, \kappa)$ be such that $\Omega_f = \pm \Omega$. This is equivalent to
  saying that there exists a constant $C \in {\mathbb R}$ such that ${\bf j}_f = C \Omega$
  (see (\ref{eq:Omega}) for the definition of ${\bf j}_f$), or in other words, that ${\bf
    j}_f \times \Omega = 0$. Now, if $\psi$ satisfies (\ref{eq:Q_adjoint_sin}), we have,
  for such a function~$f$:
  \begin{eqnarray*}
    \int_{\theta,\kappa} {\mathcal Q}_{\Omega} (f) \, \psi \, d\theta d\kappa 
    &=& \int_{\theta,\kappa} f \, Q_\Omega^* (\psi) \,d\theta d\kappa \\    &=& \beta
    \int_{\theta,\kappa} \!\!f \vec{\tau}(\theta) \times \Omega\, d\theta d\kappa = \beta
    \,  {\bf j}_f \times \Omega = 0,
  \end{eqnarray*}
  and $\psi$ is a GCI associated to $\Omega$.
  
  \noindent
  Reciprocally, if $\psi_\Omega$ is a GCI associated to $\Omega$, we have:
  \begin{equation*}
    \int_{\theta,\kappa} {\mathcal Q}_\Omega(f) \,  \psi_\Omega \, d\theta \, d\kappa = 0 =    
    \int_{\theta,\kappa} f {\mathcal Q}_\Omega^*(\psi_\Omega) \, d\theta d\kappa 
  \end{equation*}
  for all $f(\theta,\kappa)$ such that ${\bf j}_f \times \Omega = 0$. We deduce that, for all $f$, 
  \begin{equation}
    \label{eq:relation}
    {\bf j}_f \times \Omega = 0 \quad \Longrightarrow \quad \int_{\theta,\kappa} f
    Q_\Omega^*(\psi_\Omega)\,d\theta d\kappa =0.
  \end{equation}
  The two expressions appearing in (\ref{eq:relation}) are linear forms acting on $f$. By
  an elementary lemma \cite{brezis_analyse_1983}, the one appearing in the right-hand side
  is proportional to the one appearing in the left-hand side, with a proportionality
  coefficient $\beta \in {\mathbb R}$. Expressing this proportionality gives:
  \begin{equation}
    \label{eq:proportionality}
    \int_{\theta,\kappa} f ( Q_\Omega^*(\psi_\Omega) - \vec{\tau} (\theta) \times \Omega)
    \,d\theta d\kappa  = 0,
  \end{equation}
  for all $f$ without any restriction. (\ref{eq:proportionality}) yields
  (\ref{eq:Q_adjoint_sin}), which concludes the proof.
\end{proof}

It remains to prove the existence of GCI's, or, in other words, to prove the existence of
solutions to equation (\ref{eq:Q_adjoint_sin}). With this aim, we use the Hilbert space
$L_\mu^2$ equipped with the scalar product $<.,.>_\mu$ defined by:
\begin{eqnarray}
  \nonumber
  & & L^2_\mu = \{f(\theta,\kappa) \;/\; \int_{\theta,\kappa} |f|^2\,\mu\,d\theta d\kappa< + \infty\}, \\
  \label{eq:sca_prod_L2_mu}
  & &<f,\,g>_\mu = \int_{\theta,\kappa} f g\,\mu\,d\theta d\kappa.
\end{eqnarray}
Below, we will also use the notation:
\begin{equation}
  \label{eq:sca_mu}
  <g>_\mu = \int_{\theta,\kappa} g(\theta,\kappa)\,\mu(\theta,\kappa)\,d\theta d\kappa.
\end{equation}
We define the hyperplane $E$:
\begin{displaymath}
  E = \{f \in L^2_\mu(\theta,\kappa) \;/\; \int_{\theta,\kappa} f\,\mu \,d\theta d\kappa = 0 \}
\end{displaymath}
and the linear operator $\mathcal{L}$:
\begin{equation}
  \label{eq:L}
  \mathcal{L}\psi = \kappa\partial_\theta \psi - \lambda\sin \theta \partial_\kappa\psi -
  \lambda\kappa\partial_\kappa\psi + \alpha^2\partial_\kappa^2\psi
\end{equation}
with domain $D(\mathcal{L})$ given by:
\begin{displaymath}
  D(\mathcal{L}) = \{ f \in L^2_\mu \; / \; \mathcal{L}f \in L^2_\mu\}.
\end{displaymath}
We have the following lemma: 

\begin{lemma}
  \label{lem:sol_eq_elliptic}
  (i) Let $\chi \in L^2_\mu$.  A necessary condition for the existence of a solution $\psi
  \in D(\mathcal{L})$ of problem
  \begin{equation}
    \label{eq:elliptic_eq}
    \mathcal{L}\psi = \chi,
  \end{equation}
  is that $\chi \in E$ or in other words, that $\chi$ satisfies the solvability condition
  $\int_{\theta,\kappa} \chi \,\mu \,d\theta d\kappa$ $=~0$.
  
  \noindent
  (ii) For all $\chi \in E$, the problem (\ref{eq:elliptic_eq}) has a unique solution
  $\psi$ in $E$. Then, all solutions to problem (\ref{eq:elliptic_eq}) are of the form
  $\psi + K$, with an arbitrary $K \in {\mathbb R}$.
\end{lemma}

\medskip
\noindent
In Appendices A1 and A2, we give two different proofs of the fact that
(\ref{eq:elliptic_eq}) is uniquely solvable in $E$. The proof in appendix A1 uses tools
from functional analysis (see also \cite{degond_large_2008}). The proof in appendix A2
uses probabilistic tools to analyze the stochastic equation associated to
(\ref{eq:elliptic_eq}) (see also \cite{cattiaux_asymptotic_2010}). Here we only prove (i)
and the last statement of (ii).

\medskip
\begin{proof}
  (i) The formal adjoint ${\mathcal L}^*$ of ${\mathcal L}$ is given by the expression
  (\ref{eq:Q}) of $Q$ in which $\overline{\theta} = 0$. Therefore, from section
  \ref{subsubsec:equilibria}, we have that ${\mathcal L}^* (\mu) = 0$. Integrating
  (\ref{eq:elliptic_eq}) against $\mu$ and using Green's formula leads to the necessary
  condition $\int_{\theta,\kappa} \chi \,\mu \,d\theta d\kappa = 0$, i.e. to the fact that
  $\chi$ must belong to $E$.

  \medskip
  \noindent
  The second part of (ii) amounts to showing that the null space of $\mathcal{L}$ reduces
  to the constant functions. Indeed, it is straightforward to see that $\mathcal{L} (1) =
  0$. To prove that the constant functions are the only elements of the null space of
  $\mathcal{L}$, we suppose that $\psi \in D( \mathcal{L} )$ such that
  $\mathcal{L}\psi=0$. Using that $<\mathcal{L}\psi,\,\psi>_\mu=0$, we find, using Green's
  formula:
  \begin{displaymath}
    \int_{\theta,\kappa} |\partial_\kappa \psi|^2 \,\mu\,d\theta d\kappa =0.
  \end{displaymath}
  Therefore, $\psi$ is independent of $\kappa$. So we can write: $\psi(\theta,\kappa)=
  \Phi(\theta)$. Using again that $\mathcal{L}\Phi=0$, we find that $\Phi$ is a constant.

  \medskip
  \noindent
  We refer to appendices A1 or A2 for the existence part of point (ii).
\end{proof}

\medskip
\noindent
The following proposition completely determines the set of GCI's associated to a vector $\Omega$.

\begin{proposition}
  For a given $\Omega \in \mathbb{S}^1$, the set $C_\Omega$ of the GCI's associated to
  $\Omega$ is a two dimensional vector space $C_\Omega = \mbox{Span} \{1, \psi_\Omega\}$
  where $\psi_\Omega$ is given by:
  \begin{equation}
    \label{eq:gci_psi}
    \psi_\Omega(\theta,\kappa) = \psi(\theta-\overline{\theta},\kappa),
  \end{equation}
  with $\overline{\theta}$ such that $\vec{\tau}(\theta)=\Omega$ and $\psi$ is the unique solution of:
  \begin{equation}
    \label{eq:psi_0_bis}
    \mathcal{L}\psi = -\sin \theta,
  \end{equation}
  belonging to the hyperplane $E$. Moreover, the function $\psi$ satisfies the property:
  \begin{equation}
    \label{eq:sym_psi_0}
    \psi(-\theta,-\kappa)=-\psi(\theta,\kappa).
  \end{equation}
\end{proposition}

\noindent
\begin{proof}
  We first note that (\ref{eq:Q_adjoint_sin}) is a linear problem and that it is enough to
  solve it for $\beta = 1$. Simple calculations show that $\psi_\Omega$ is a solution to
  (\ref{eq:Q_adjoint_sin}) if and only if there exists a function $\psi$ such that
  $\psi_\Omega (\theta) = \psi (\theta - \overline{\theta})$ with $\psi$ a solution of
  (\ref{eq:psi_0_bis}). This shows (\ref{eq:gci_psi}).

  To show the existence and uniqueness of a solution $\psi$ to (\ref{eq:psi_0_bis}) in
  $E$, it is enough to check that the right-hand side of (\ref{eq:psi_0_bis}) belongs to
  $E$ i.e.  satisfies the compatibility condition $\int_{\theta,\kappa} \chi \,\mu
  \,d\theta d\kappa = 0$. But this follows readily by oddness considerations.  Moreover,
  noting that the operator ${\mathcal L}$ is invariant under the transformation $(\theta,
  \kappa) \to (-\theta, -\kappa)$, (\ref{eq:sym_psi_0}) follows from the uniqueness of the
  solution.
  
  Again, by the uniqueness in $E$ and by the second part of Lemma
  \ref{lem:sol_eq_elliptic} (ii), all solutions to (\ref{eq:psi_0_bis}) consist of linear
  combinations of $\psi$ and of a constant function. It follows that the set of GCI's
  associated to $\Omega$ is the two-dimensional vector space spanned $C_\Omega = Span\{1,
  \psi_\Omega\}$. This ends the proof.
\end{proof}

\subsection{Limit $\varepsilon\rightarrow0$}
\label{subseq:limit}

Since we know the equilibria and GCI's of the operator $Q$, we can give a formal proof of
theorem \ref{thm:formal_cvgnce}.

\begin{proof2} {\bf of Theorem \ref{thm:formal_cvgnce}.}
  If we suppose that $f^\varepsilon$ converges (weakly) to $f^0$ as $\varepsilon\rightarrow0$
  we first have:
  \begin{displaymath}
    Q(f^0) = 0,
  \end{displaymath}
  which means that $f^0$ is an equilibrium. Thanks to section \ref{subsubsec:equilibria},
  $f^0$ can be written as:
  \begin{displaymath}
    f^0 = \rho^0 \mathcal{M}_{\Omega^0}(\theta) \mathcal{N}(\kappa),
  \end{displaymath}
  with $\mathcal{M}$ and $\mathcal{N}$ defined in (\ref{eq:Von})~(\ref{eq:Gauss}). The
  mass $\rho^0(t,{\bf x})$ and the direction of the flux $\Omega^0(t,{\bf x})$ are the two
  remaining unknowns.

  In order to find the system of equations which determines the evolution of $\rho^0$ and
  $\Omega^0$, we first integrate (\ref{eq:f_ep2}) with respect to $(\theta,\kappa)$. We
  find the mass conservation equation:
  \begin{displaymath}
    \partial_t\rho^\varepsilon + \nabla_{{\bf x}}\cdot{\bf j}^\varepsilon = 0,
  \end{displaymath}  
  with
  \begin{displaymath}
    {\bf j}^\varepsilon = \int_{\theta,\kappa}\vec{\tau}(\theta) f^\varepsilon\,d\theta d\kappa.
  \end{displaymath}
  In the limit $\varepsilon\rightarrow0$, this gives:
  \begin{displaymath}
    {\bf j}^\varepsilon \stackrel{\varepsilon \rightarrow 0}{\longrightarrow} {\bf
      j}^0=c_1 \rho^0 \Omega^0, 
  \end{displaymath}
  with the constant $c_1$ given by:
  \begin{equation}
    \label{eq:c_1}
    c_1 = \int_{\theta} \cos \theta\,\mathcal{M}(\theta)\,d\theta =
    \frac{I_1(\frac{\lambda^2}{\alpha^2})}{I_0(\frac{\lambda^2}{\alpha^2})} .
  \end{equation}
  Therefore we deduce that $\rho^0$ and $\Omega^0$ obey the following mass conservation
  equation:
  \begin{displaymath}
    \partial_t \rho^0 + c_1 \nabla_{{\bf x}}\cdot (\rho^0\Omega^0) = 0.
  \end{displaymath}
  
  \medskip In order to fully determine the evolution of $\rho^0$ and $\Omega^0$, we need
  to find a second equation. For this purpose, we integrate (\ref{eq:f_ep2}) against the
  generalized collisional invariant $\psi_{\Omega^\varepsilon}$ (\ref{eq:gci_psi}), with
  $\Omega^\varepsilon = \Omega_{f^\varepsilon}$. This leads to:
  \begin{displaymath}
    \int_{\theta,\kappa} (\partial_t f^\varepsilon + \vec{\tau}(\theta)\cdot\nabla_x
    f^\varepsilon)\psi_{\Omega^\varepsilon}\,d\theta d\kappa = 0.
  \end{displaymath}
  In the limit $\varepsilon\rightarrow0$, we find :
  \begin{equation}
    \label{eq:cinetique_inv2}
    \int_{\theta,\kappa} \partial_t (\rho^0 \mathcal{M}_{\Omega^0} \mathcal{N}) \,
    \psi_{\Omega^0}\,d\theta d\kappa + \int_{\theta,\kappa} \vec{\tau}(\theta) \cdot
    \nabla_{{\bf x}}(\rho^0\mathcal{M}_{\Omega^0} \mathcal{N}) \, \psi_{\Omega^0}\,d\theta
    d\kappa =0.
  \end{equation}
  For clarity, we drop the exponent '0' and write $(\rho,\Omega)$ for $(\rho^0,\Omega^0)$
  in the discussion below. Using polar coordinates for
  $\Omega=\vec{\tau}(\overline{\theta})=(\cos \theta, \sin \theta)$, elementary
  computations show that:
  \begin{eqnarray*}
    \partial_t (\rho \mathcal{M}_{\overline{\theta}}) +
    \vec{\tau}(\theta)\cdot\nabla_{{\bf x}}(\rho\mathcal{M}_{\overline{\theta}})
    &=& \partial_t \rho\,\mathcal{M}_{\overline{\theta}}
    + \rho \mathcal{M}_{\overline{\theta}}\, \frac{\lambda^2}{\alpha^2}
    \sin(\theta-\overline{\theta}) \partial_t \overline{\theta} \\
    && \hspace{-1.5cm} +\, \vec{\tau}(\theta)\cdot\big(\nabla_{\!{\bf x}}\rho
    \mathcal{M}_{\overline{\theta}} + \rho \mathcal{M}_{\overline{\theta}}
    \,\frac{\lambda^2}{\alpha^2} \sin(\theta-\overline{\theta}) \nabla_{{\bf x}}
    \overline{\theta}\big).
  \end{eqnarray*}
  Therefore, equation (\ref{eq:cinetique_inv2}) leads to:
  \begin{eqnarray*}
    && \int_{\theta,\kappa} \partial_t \rho\,\mathcal{M}_{\overline{\theta}} \mathcal{N}\,
    \psi_{\overline{\theta}}\,d\theta d\kappa \\
    &+& \frac{\lambda^2}{\alpha^2} \int_{\theta,\kappa} \rho\,
    \mathcal{M}_{\overline{\theta}} \mathcal{N}\,
    \sin(\theta-\overline{\theta})\, \partial_t\overline{\theta}\,
    \psi_{\overline{\theta}}\,d\theta d\kappa \\
    &+& \int_{\theta,\kappa} \vec{\tau}(\theta)\cdot\big(\nabla_{\!{\bf x}}\rho\,
    \mathcal{M}_{\overline{\theta}} \mathcal{N}\, \psi_{\overline{\theta}}\big) \,d\theta
    d\kappa \\
    &+& \frac{\lambda^2}{\alpha^2} \int_{\theta,\kappa}
    \vec{\tau}(\theta)\cdot\big(\rho\,\mathcal{M}_{\overline{\theta}} \mathcal{N}\,
    \sin(\theta-\overline{\theta})\, \nabla_{{\bf x}} \overline{\theta}\,
    \psi_{\overline{\theta}}\big)\,d\theta d\kappa = 0.
  \end{eqnarray*}
  This equation can be simplified using the symmetry satisfied by $\psi$
  (\ref{eq:sym_psi_0}). We treat each term separately. First, we have:
  \begin{eqnarray}
    \nonumber 
    X_1 &=& \int_{\theta,\kappa} \partial_t \rho\,\mathcal{M}_{\overline{\theta}}
    \mathcal{N}\, \psi_{\overline{\theta}}\,d\theta d\kappa \\
    &=& \partial_t \rho \int_{\theta,\kappa}\,\mathcal{M}(\theta-\overline{\theta})
    \mathcal{N}(\kappa)\, \psi(\theta-\overline{\theta},\kappa)\,d\theta d\kappa =
    0, \label{eq:X_1}
  \end{eqnarray}
  because $\mathcal{M}(\theta) \mathcal{N}(\kappa)$ is an even function of the pair
  $(\theta, \kappa)$ and $\psi(\theta,\kappa)$ is odd. For the second term, we use the
  change of unknowns $\theta'=\theta-\overline{\theta}$ and get:
  \begin{eqnarray}
    \nonumber
    X_2 &=& \frac{\lambda^2}{\alpha^2}\, \rho\, \partial_t\overline{\theta}\,
    \int_{\theta',\kappa} \mathcal{M}(\theta') \mathcal{N}(\kappa) \sin \theta'\,
    \psi(\theta',\kappa)\,d\theta' d\kappa \\
    &=& \frac{\lambda^2}{\alpha^2}\, \rho\, \partial_t\overline{\theta}\,
    \gamma_1, \label{eq:X_2}
  \end{eqnarray}
  with
  \begin{equation}
    \label{eq:gamma1}
    \gamma_1 = < \sin \theta\, \psi >_\mu
  \end{equation}
  using the notation (\ref{eq:sca_mu}). For the third term, we find:
  \begin{eqnarray*}
    X_3 &=&  \nabla_{\!{\bf x}}\rho \cdot \int_{\theta,\kappa}
    \vec{\tau}(\theta)\,\mathcal{M}_{\overline{\theta}} \mathcal{N}\,
    \psi_{\overline{\theta}} \,d\theta d\kappa \\
    &=& \nabla_{\!{\bf x}}\rho \cdot \int_{\theta,\kappa}
    \vec{\tau}(\theta+\overline{\theta})\,\mathcal{M}(\theta)\mathcal{N}(\kappa)\,\psi(\theta,\kappa)\,d\theta
    d\kappa \\
    &=& \nabla_{\!{\bf x}}\rho \cdot \int_{\theta,\kappa} \left(
      \begin{array}{c}
        \cos \theta \cos \overline{\theta} - \sin \theta \sin \overline{\theta} \\
        \sin \theta \cos \overline{\theta} + \cos \theta \sin \overline{\theta}
      \end{array}
    \right) \mathcal{M}(\theta)\mathcal{N}(\kappa)\,\psi(\theta,\kappa)\,d\theta d\kappa.
  \end{eqnarray*}
  Once again, using the symmetry satisfied  by $\psi$, we find:
  \begin{displaymath}
    X_3 = \gamma_1 \nabla_{\!{\bf x}}\rho \cdot \left(
      \begin{array}{c}
        - \sin \overline{\theta} \\
        \cos \overline{\theta}
      \end{array}
    \right),
  \end{displaymath}
  with $\gamma_1$ defined in (\ref{eq:gamma1}). If we denote by
  $\vec{\tau}(\overline{\theta})^\perp=\Omega^\perp$ the orthogonal vector to
  $\vec{\tau}(\overline{\theta})$:
  \begin{displaymath}
    \vec{\tau}(\overline{\theta})^\perp = \Omega^\perp = \left(
      \begin{array}{c}
        - \sin \overline{\theta} \\
        \cos \overline{\theta}
      \end{array}
    \right),
  \end{displaymath}
  we finally get:
  \begin{equation}
    \label{eq:X_3}
    X_3 = \gamma_1 \nabla_{\!{\bf x}}\rho \cdot \vec{\tau}(\overline{\theta})^\perp.
  \end{equation}
  For the last term, we have:
  \begin{eqnarray}
    \nonumber
    X_4 &=& \frac{\lambda^2}{\alpha^2}\,\rho\,\nabla_{{\bf x}}\overline{\theta} \cdot \int_{\theta,\kappa}
    \vec{\tau}(\theta) \mathcal{M}_{\overline{\theta}}(\theta) \mathcal{N}(\kappa)
    \sin(\theta-\overline{\theta}) \psi_{\overline{\theta}}(\theta,\kappa)\,d\theta
    d\kappa \\
    \nonumber
    &=& \frac{\lambda^2}{\alpha^2}\,\rho\,\nabla_{{\bf x}}\overline{\theta} \cdot
    \int_{\theta,\kappa} \vec{\tau}(\theta+\overline{\theta}) \mathcal{M}_{0}(\theta)
    \mathcal{N}(\kappa) \sin \theta\, \psi(\theta,\kappa)\,d\theta d\kappa \\
    &=& \frac{\lambda^2}{\alpha^2}\,\gamma_2\,\rho\,\nabla_{{\bf x}}\overline{\theta}
    \cdot \vec{\tau}(\theta), \label{eq:X_4}
  \end{eqnarray}
  with
  \begin{displaymath}
    \gamma_2 = < \cos \theta\, \sin \theta\, \psi>_\mu.
  \end{displaymath}
  Combining (\ref{eq:X_1}),~(\ref{eq:X_2}),~(\ref{eq:X_3}) and (\ref{eq:X_4}) yields:
  \begin{equation}
    \label{eq:theta_t}
    \gamma_1\,\frac{\lambda^2}{\alpha^2}\, \rho\, \partial_t\overline{\theta} + \gamma_1
    \nabla_{\!{\bf x}}\rho \cdot \vec{\tau}(\overline{\theta})^\perp \,+\, \,\gamma_2
    \frac{\lambda^2}{\alpha^2}\,\rho\,\nabla_{{\bf x}}\overline{\theta} \cdot
    \vec{\tau}(\theta) = 0.
  \end{equation}
  Using again the unit vector $\Omega=\vec{\tau}(\overline{\theta})$, elementary
  computations show that:
  \begin{displaymath}
    \partial_t \Omega = \partial_t \overline{\theta} \, \Omega^\perp \qquad \text{ and }
    \qquad (\Omega\cdot\nabla_{{\bf x}})\Omega = (\Omega^\perp\otimes\Omega)  \nabla_{{\bf
        x}} \overline{\theta}.
  \end{displaymath}
  Therefore, multiplying equation (\ref{eq:theta_t}) by $\Omega^\perp$ leads to:
  \begin{displaymath}
    \rho\, \partial_t \Omega \,+\, \frac{\alpha^2}{\lambda^2}\, (\nabla_{\!{\bf x}}\rho
    \cdot \Omega^\perp) \Omega^\perp \,+\, \frac{\gamma_2}{\gamma_1}\;\rho\, (\Omega \cdot
    \nabla_{{\bf x}}) \Omega = 0.
  \end{displaymath}
  This finally leads to:
  \begin{equation}
    \label{eq:dt_omega}
    \rho\, \partial_t \Omega \,+\,c_2\;\rho\,(\Omega\cdot \nabla_{{\bf x}}) \Omega \,+\,
    \frac{\alpha^2}{\lambda^2}\, (\mbox{Id} - \Omega\otimes\Omega)\nabla_{\!{\bf x}}\rho =
    0,
  \end{equation}
  with
  \begin{equation}
    \label{eq:c_2}
    c_2 = \frac{\gamma_2}{\gamma_1} = \frac{<\sin \theta \cos \theta\,\psi>_\mu}{<\sin
      \theta\,\psi>_\mu},
  \end{equation}
  which end the proof. 
\end{proof2}

\section{Properties of the macroscopic system}
\label{sec:macro_eq}
\setcounter{equation}{0}

\subsection{Hyperbolicity}
\label{subseq:hyperbolicity}

The macroscopic system (\ref{eq:final_system}) arising from the PTWA dynamics has the same
form as the system found in \cite{degond_continuum_2008} for the macroscopic limit of the
Vicsek model.  Indeed, if we define the diffusion coefficient $d$ as:
\begin{displaymath}
  d = \frac{\alpha^2}{\lambda^2},
\end{displaymath}
then the coefficient $c_1$ given by (\ref{eq:c_1}) and the coefficient
$\frac{\alpha^2}{\lambda^2}$ in front of the pressure term in (\ref{eq:dt_omega}) are
exactly the same in the two systems. Only the coefficient $c_2$ given by (\ref{eq:c_2})
differs from that of \cite{degond_continuum_2008}. Thus, the study of the hyperbolicity of
system (\ref{eq:final_system}) is completely similar to the one conducted for the Vicsek
model in \cite{degond_continuum_2008,motsch_numerical_2010}. We briefly summarize the
analysis here. Using the geometric constraint $|\Omega|=1$, we can parametrize the
direction of the flux $\Omega$ in polar coordinates: $\Omega=(\cos \theta,\,\sin \theta)$
with $\theta \in ]-\pi,\pi]$. In order to look at the wave propagating in the
$x$-direction, we suppose that $\rho$ and $\Omega$ are independent of $y$. Therefore,
under this assumption, the system (\ref{eq:final_system}) reduces to:
\begin{eqnarray*}
  && \ds \partial_t \rho + c_1 \partial_x \,(\rho \cos \theta)  \;\; = \;\; 0, \\
  && \ds \partial_t \theta + c_2 \cos \theta \partial_x \theta -
  \frac{\alpha^2}{\lambda^2}  \frac{\sin \theta}{\rho} \partial_x \rho \;\; = \;\; 0.
\end{eqnarray*}
The characteristic velocities of this system are given by:
\begin{displaymath}
  \gamma = \frac{1}{2} \left[ (c_1+c_2)\cos \theta \pm \sqrt{(c_1-c_2)^2 \cos^2 \theta
      + 4 c_1 \frac{\alpha^2}{\lambda^2} \sin^2 \theta}\right].
\end{displaymath}
The system is therefore \emph{hyperbolic} since the characteristic velocities are real.

\subsection{Numerical computations of $\psi$}
\label{subseq:numeric}


In order to compute the macroscopic coefficient $c_2$ (\ref{eq:c_2}), we first need to
calculate the generalized collisional invariant $\psi$ (\ref{eq:gci_psi}). With this aim,
we introduce a weak formulation of the equation satisfied by $\psi$. In the Hilbert space
$L_\mu^2(\mathbb{S}^1\times\mathbb{R})$, the function $\psi$ satisfies:
\begin{equation}
  \label{eq:weak_gci}
  <\mathcal{L}\psi,\,\varphi>_\mu = -<\sin \theta,\,\varphi>_\mu\;, \quad \forall\,\varphi \in L_\mu^2,
\end{equation}
where the scalar product $<.,.>_\mu$ is defined in (\ref{eq:sca_prod_L2_mu}) and the
operator $\mathcal{L}$ in (\ref{eq:L}). To approximate the solution $\psi$ numerically, we
use a Galerkin method. It consists in solving the weak formulation (\ref{eq:weak_gci}) for
all the functions $\varphi$ in a subspace $V$ of $L_\mu^2$ of finite dimension. To
construct such a subspace $V$, we use a Hilbert basis of $L_\mu^2$. For this purpose, we
consider the following functions:
\begin{displaymath}
  \varphi_m(\theta) = \frac{\expo^{im \theta}}{\sqrt{2\pi\mathcal{M}(\theta)}} \quad
  \text{ , } \quad  P_n(\kappa) =  \frac{H_n\left(\frac{\sqrt{\lambda}}{\alpha}
      \kappa\right)}{\sqrt{n!}},
\end{displaymath}
where $\mathcal{M}$ is defined in (\ref{eq:Von}) and $H_n$ is the $\text{n}^{\text{th}}$
Hermite polynomial. We can easily prove that the family $\{\varphi_m P_n\}_{m,\,n\geq0}$
is a Hilbert basis of $L_\mu^2$. Then, for any odd positive integers $m$ and any positive
integer $n$, we define the vector space $V_{m,n}$:
\begin{displaymath}
  V_{m,n} = \text{Span}\{\varphi_j P_k \; / \; |j|\,\leq m  \; ,\; 0\leq\, k\, \leq n\}.
\end{displaymath}
The Galerkin method consists in finding $\psi_{m,n} \in V_{m,n}$ such that equation
(\ref{eq:weak_gci}) is satisfied for every $\varphi \in V_{m,n}$:
\begin{equation}
  \label{eq:weak_gci_mn}
  <\mathcal{L}\psi_{m,n},\,\varphi>_\mu = -<\sin \theta,\,\varphi>_\mu \; , \quad
  \forall\,\varphi \in V_{m,n}.
\end{equation}
We can decompose $\psi$ as:
\begin{equation}
  \label{eq:dec_psi}
  \psi_{m,n}(\theta,\kappa) = \sum_{|j|<m , 0\leq k\leq n} C_j^k \varphi_m(\theta) P_n(\kappa),
\end{equation}
where $C_j^k$ are complex coefficients given by:
\begin{displaymath}
  C_j^k = <\psi_{m,n},\,\varphi_j\,P_k>_\mu.
\end{displaymath}
We store the coefficients $\{C_j^k\}_{|j|\leq m,0\leq k\leq n}$ in a matrix $X$ such that:
\begin{equation}
  \label{eq:matrix_X}
  X(j,k) = C_j^k.
\end{equation}
We call the matrix $X$ the matrix representation of $\psi_{m,n}$ in $V_{m,n}$. We want to
transform the problem satisfied by $\psi_{m,n}$ (\ref{eq:weak_gci}) into a matrix equation
for $X$. With this aim, we define several matrices.

\begin{definition}
  We define the matrices $L_{-1}$ and $L_{+1}$ by:
  \begin{equation}
    \label{eq:L_mp}
    L_{-1}  = \left[
      \begin{array}{ccccc}
        0 &  &   & \\
        1 & 0 &  & \\
        & \ddots & \ddots &  \\
        &  & 1 & 0
      \end{array}
    \right], \quad
    L_{+1}  = \left[
      \begin{array}{cccc}
        0 & 1 &  & \\
        & \ddots & \ddots &  \\
        &   & 0 & 1 \\
        &   &   & 0
      \end{array}
    \right].
  \end{equation}
  and the diagonal matrices:
  \begin{eqnarray*}
    D_1 &=& \text{diag}(-\!m,\, \dots, -\!1,\, 0,\, 1,\, \dots,\, m)\\
    D_2 &=& \text{diag}(0,\, 1,\, 2,\, \dots\, n).
  \end{eqnarray*}
\end{definition}

\noindent
Using the matrices defined above, we can convert the equation satisfied by $\psi_{m,n}$
(\ref{eq:weak_gci_mn}) into a matrix equation for $X$.

\begin{proposition}
  \label{ppo:psi_X}
  Let $\psi_{m,n}\in V_{m,n}$ the solution of (\ref{eq:weak_gci}) in $V_{m,n}$. Its matrix
  representation $X=\{C_j^k\}_{|j|\leq m,0\leq k\leq n}$ in the Hilbert basis $\{\varphi_m
  P_n\}$ satisfies:
  \begin{equation}
    \label{eq:X_equation}
    \beta_1 M_1 X N_1 \,+\, \beta_2 M_2 XN_2 \,-\, \lambda X D_2 = B
  \end{equation}
  with
  \begin{equation}
    \label{eq:MN}
    \begin{array}{rclcrcl}
      \beta_1 &=& \frac{i\alpha}{\sqrt{\lambda}}  &\quad,\quad& \beta_2 &=&
      \frac{i\lambda\sqrt{\lambda}}{4\alpha},\\
      M_1 &=& D_1  &\quad,\quad& N_1 &=& \sqrt{D_2}L_{-1} + L_{+1} \sqrt{D_2} ,\\
      M_2 &=& L_{-1}-L_{+1} &\quad,\quad& N_2 &=& \sqrt{D_2}L_{-1} - L_{+1} \sqrt{D_2}.
    \end{array}
  \end{equation}
  and $B$ the matrix representation of $-\sin \theta$ in $V_{m,n}$ given by:
  \begin{displaymath}
    B(j,k) = \left\{
      \begin{array}{lc}
        \frac{i}{2\sqrt{I_o(\frac{\lambda^2}{\alpha^2})}}\,\left(
          I_{|j-1|}\left(\frac{\lambda^2}{2\alpha^2}\right) -
          I_{|j+1|}\left(\frac{\lambda^2}{2\alpha^2}\right) \right) & \text{if } k=0, \\
        0 & \text{otherwise,}
      \end{array}
    \right.
  \end{displaymath}
  where $I_j$ is the modified Bessel function of order $j$.
\end{proposition}

\noindent
Since the demonstration of proposition \ref{ppo:psi_X} is only a matter of computations,
we postpone the proof to appendix B. To solve (\ref{eq:X_equation}), we transform the
linear equation (\ref{eq:X_equation}) into a linear system that we invert numerically.
This eventually allows us to construct $\psi_{m,n}$ using (\ref{eq:dec_psi}).

On figure \ref{fig:psi} (left), we display an example of an approximate solution
$\psi_{m,n}$ of the GCI $\psi$ for $\lambda=1$ and $\alpha=1$. We also estimate
$\mathcal{L}\psi_{m,n}$ numerically using a finite difference method (figure
\ref{fig:psi}, right). The figure clearly suggests that $\mathcal{L}\psi_{m,n}$ is close
to $-\sin \theta$, providing a qualitative check of the accuracy of the computation.  To
make this assessment more quantitative, we compute the residual $|\mathcal{L}\psi_{m,n} +
\sin \theta|_\infty$ for different values of $(\lambda,\,\alpha)$ on figure
\ref{fig:errorPsi}. As we can see, the residual gets larger when $\alpha$ increases and
gets smaller when $\lambda$ increases.

\begin{figure}[ht]
  \centering
  \includegraphics[scale=.5]{psi.eps}
  \includegraphics[scale=.5]{Lpsi.eps}
  \caption{Left figure: the generalized collisional invariant $\psi_{m,n}$ for $\lambda=1$
    and $\alpha=1$ computed using $m=30$ and $n=61$. Right figure: we compute
    $\mathcal{L}\psi_{m,n}$ using a finite difference method with $\Delta\theta = .2$ and
    $\Delta\kappa=.2$. We clearly recover the function $-\sin \theta$ (see figure
    \ref{fig:errorPsi} for a more detailed comparison).}
  \label{fig:psi}
\end{figure}

\begin{figure}[ht]
  \centering
  \includegraphics[scale=.5]{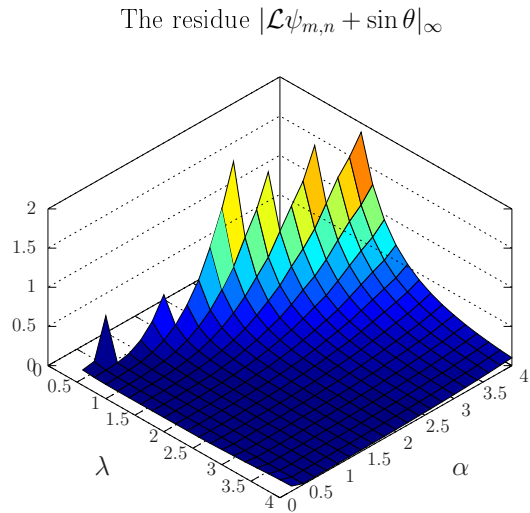}
  \caption{The residual $|\mathcal{L}\psi_{m,n}+\sin \theta|_\infty$ estimated on the
    interval $(\theta,\kappa)\in[-\pi,\pi]\times[-5,5]$ for different values of
    $(\lambda,\alpha)$. $\psi_{m,n}$ is computed as in figure \ref{fig:psi} (left) and
    $\mathcal{L}\psi_{m,n}$ is computed using a finite difference scheme (with
    $\Delta\theta=\Delta\kappa=.2$). The residual increases with $\alpha$ and decreases
    with $\lambda$.}
  \label{fig:errorPsi}
\end{figure}

\subsection{Computation of the coefficient $c_2$}
\label{subseq:coef}

Once we have computed the generalized collisional invariant $\psi$, we can calculate the
coefficient $c_2$ using (\ref{eq:c_2}). On figure \ref{fig:c2_PTWA_Vic}, we fix the the
parameter $\lambda=1$ and we compute the value of $c_2$ for different values of $\alpha$
(we still use $m=30$ and $n=61$ to get a numerical approximation of $\psi_{m,n}$). In the
same graph, we add show the coefficient $c_2$ of the Vicsek model
\cite{degond_continuum_2008,motsch_numerical_2010} for $d=\frac{\alpha^2}{\lambda^2}$. The
relative error between the two curves is very small (around $5\%$). This similarity
between the two curves shows a strong connexion between the PTWA model and the Vicsek
model. Work is in progress to study the link between the two models more deeply.

\begin{figure}[ht]
  \centering
  \includegraphics[scale=.5]{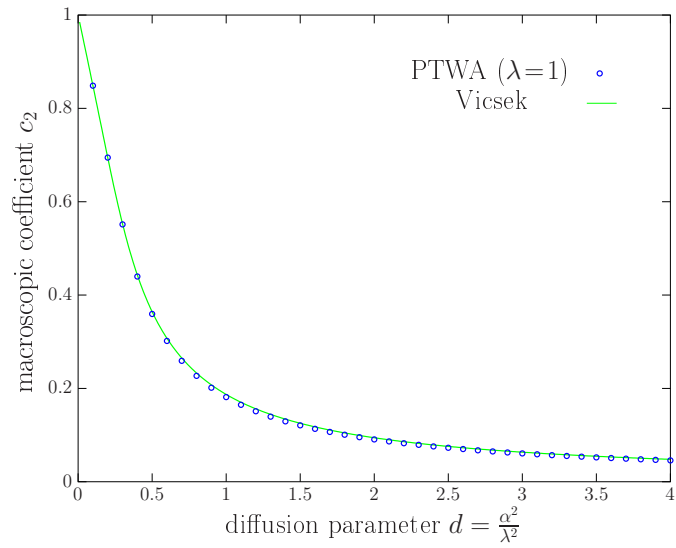}
  \caption{The coefficient $c_2$ in the PTWA model (\ref{eq:c_2}) computed for $\lambda=1$
    and different values of $\alpha$ (blue) and the coefficient $c_2$ in the Vicsek model
    (green). The relative error between the two curves is around $5\%$.}
  \label{fig:c2_PTWA_Vic}
\end{figure}

\clearpage

\section{Conclusion}
\label{sec:conclusion}
\setcounter{equation}{0}

In this work, we have introduced a new Individual-Based Model describing the displacement
of individuals which tend to align with theirs neighbors. This model, called 'Persistent
Turning Walker model with Alignment' (PTWA), is a combination of the phenomenological
Vicsek alignment model \cite{vicsek_novel_1995} with the experimentally derived PTW model
of fish displacement \cite{gautrais_analyzing_2009}. We have established the macroscopic
limit of this model within a hydrodynamic scaling where the radius of interaction of the
agents is tied to the microscopic scale. The derivation uses a new notion of 'Generalized
Collisional Invariant' developed earlier in \cite{degond_continuum_2008}. The numerical
computations of the coefficients involved in this macroscopic model have shown that there
are important similarities between the PTWA model and the Vicsek model at large scale.

The present work proves that the addition of a local alignment rule in the PTW model
changes drastically the large-scale dynamics as compared to the PTW model without
alignment interaction. Indeed, while the PTW model without alignment is diffusive at large
scales, the PTWA model becomes hyperbolic, of hydrodynamic type. As a summary, local
alignment generates macroscopic convection.

In future work, the relation between the PTWA and Vicsek dynamics will be further explored,
both at the microscopic and macroscopic levels. This ensemble of models forms a complex
hierarchy. Numerical simulations and comparisons over a wide range of parameters will be
performed to better understand the relations between these models.

Many questions concerning the derivation of macroscopic models remain open in this
context. One possible route is to explore what the macroscopic limit of the PTWA model
becomes when an attraction-repulsion rule is added. More generally, it may be possible to
classify the different types of Individual-Based Models by looking at their corresponding
macroscopic limits. Another direction is to quantify how close the macroscopic model is to
the corresponding microscopic model. In particular, the question of determining what
minimal number of individuals is required for the macroscopic description to be valid is
of crucial importance.  All these questions call for deeper numerical studies which will
permit to understand when the microscopic and macroscopic descriptions are similar and
when they are not.


\clearpage

\newpage


\section*{Appendix A1:  Proof of lemma \ref{lem:sol_eq_elliptic} (ii) (functional analytic proof)}

\noindent {\bf Proof.} First, we prove the uniqueness of the solution of
(\ref{eq:elliptic_eq}) in $E$. Indeed, we have shown in section
\ref{subsubseq:collision_invar} that the null space $\ker({\mathcal L})$ of ${\mathcal L}$
consists of the constant functions. Therefore, $\ker ({\mathcal L}) \cap E = \{ 0 \}$,
which shows the uniqueness of the solutions of (\ref{eq:elliptic_eq}) in $E$.

To prove the existence of a solution of (\ref{eq:elliptic_eq}), we first consider a
slightly modified version of equation (\ref{eq:elliptic_eq}): for a given $\varepsilon>0$, we want
to solve
\begin{equation}
  \label{eq:L_ep}
  -\varepsilon \psi + \mathcal{L}\psi \; =\; \chi.
\end{equation}
Thanks to this modification, we have the inequality:
\begin{displaymath}
  <\varepsilon \psi - \mathcal{L}\psi,\,\psi>_\mu = \varepsilon |\psi|_\mu^2 + \alpha^2
  |\partial_\kappa \psi|_\mu^2 \geq \varepsilon |\psi|_\mu^2.
\end{displaymath}
Therefore the operator $\varepsilon Id - \mathcal{L}$ is coercive, so we can apply the
theorem of J. L. Lions in \cite{lions_equations_1961} which gives a weak solution
$\psi_\varepsilon$ in $E$ of the problem (\ref{eq:L_ep}).

To find a solution of $\mathcal{L}\psi=\chi$, we need to extract a convergent subsequence
of $\{\psi_\varepsilon\}_{\varepsilon>0}$ when $\varepsilon$ goes to zero. The limit will
satisfy (\ref{eq:elliptic_eq}). Since $E$ is an Hilbert space, it remains to prove that
the family $\{\psi_\varepsilon\}_{\varepsilon>0}$ is bounded in $E$. For that, we proceed
by contradiction. If the family $\{\psi_\varepsilon\}_\varepsilon$ is not bounded in $E$
as $\varepsilon$ tends to $0$, there exists a subsequence $\varepsilon_n$ such that:
\begin{displaymath}
  |\psi_{\varepsilon_n}|_\mu \stackrel{n \rightarrow \infty}{\longrightarrow} +\infty
  \qquad,\qquad \varepsilon_n \stackrel{n \rightarrow \infty}{\longrightarrow} 0.
\end{displaymath}
To simplify the notations, we use the subscript $\varepsilon$ for $\varepsilon_n$ in the
following. Defining the functions:
\begin{equation}
  \label{eq:U_ep}
  U_\varepsilon = \frac{\psi_\varepsilon}{N_\varepsilon} 
\end{equation}
with $N_\varepsilon = |\psi_\varepsilon|_\mu$, we have that:
\begin{displaymath}
  -\varepsilon U_\varepsilon + \mathcal{L}U_\varepsilon = \frac{\chi}{N_\varepsilon}.
\end{displaymath}
Since the sequence $\{U_\varepsilon\}_\varepsilon$ is bounded ($|U_\varepsilon|_\mu=1$),
we can extract a weakly convergent subsequence (denoted by $\varepsilon$ once again) such
that:
\begin{displaymath}
  U_\varepsilon \stackrel{\varepsilon \rightarrow 0}{\rightharpoonup} U_0 \quad \text{
    weakly in } L_\mu^2.
\end{displaymath}
In particular, since $N_\varepsilon \stackrel{\varepsilon \rightarrow 0}{\longrightarrow}
+\infty$, we have that $\mathcal{L}U_0=0$ and therefore by uniqueness $U_0=0$. This means
that $U_\varepsilon$ converges weakly to zero. We will obtain a contradiction with the
fact $|U_\varepsilon|_\mu=1$ if we prove that $U_\varepsilon$ converges strongly to zero.

To prove the strong convergence of $U_\varepsilon$, we decompose the functions
$U_\varepsilon$ in two parts. For that, we introduce the vector space $L$:
\begin{displaymath}
  L = \{\Phi\in L^2(\mathbb{S}^1)\; /\; \int_{\theta} \Phi(\theta)\,\mathcal{M}(\theta)\,d\theta = 0\}.
\end{displaymath}
It is easy to see that $L \subset E$. We denote by $L^\perp$ the orthogonal space of $L$ such that:
\begin{displaymath}
  E = L \stackrel{\perp}{\oplus} L^\perp.
\end{displaymath}
We can decompose the sequence $U_\varepsilon$ as $U_\varepsilon = \Phi_\varepsilon +
v_\varepsilon$ with $\Phi_\varepsilon\in L$ and $v_\varepsilon \in L^\perp$. First, we are
going to prove that $v_\varepsilon$ converges to zero using that $\mathcal{L}$ is coercive
on $L^\perp$. Taking the scalar product of the equation (\ref{eq:U_ep}) against
$U_\varepsilon$, we find:
\begin{displaymath}
  -\varepsilon |U_\varepsilon|_\mu^2 + <\mathcal{L}U_\varepsilon,\,U_\varepsilon>_\mu =
  \frac{1}{N_\varepsilon} <\chi,\,U_\varepsilon>.
\end{displaymath}
Therefore, at the limit $\varepsilon\rightarrow0$, we have:
\begin{displaymath}
  <\mathcal{L}U_\varepsilon,\,U_\varepsilon>_\mu \stackrel{\varepsilon \rightarrow 0}{\longrightarrow} 0.
\end{displaymath}
Since we have the equality $<\mathcal{L}U_\varepsilon,\,U_\varepsilon>_\mu =
-\alpha^2|\partial_\kappa U_\varepsilon|_\mu^2$ (\ref{eq:entropy}) and $\partial_\kappa
U_\varepsilon = \partial_\kappa v_\varepsilon$, we obtain that:
\begin{equation}
  \label{eq:d_k_v_ep}
  |\partial_\kappa v_\varepsilon|_\mu^2 \stackrel{\varepsilon \rightarrow 0}{\longrightarrow} 0.
\end{equation}
Then we use the Poincar\'{e} inequality for Gaussian measures \cite{gross_logarithmic_1993}:
\begin{equation}
  \label{eq:Poincare_Gauss}
  \int_\kappa |f - \overline{f}|^2\,\mathcal{N}\,d\kappa \leq C \int_\kappa
  |\partial_\kappa f|^2\,\mathcal{N}\,d\kappa,
\end{equation}
with $C$ a positive constant and $\overline{f}$ the mean of $f$ defined as:
\begin{displaymath}
  \overline{f} = \int_\kappa f(\kappa)\,\mathcal{N}\,d\kappa.
\end{displaymath}
Applying the Poincar\'{e} inequality (\ref{eq:Poincare_Gauss}) to $v_\varepsilon$ leads
to:
\begin{eqnarray}
  \nonumber 
  |\partial_\kappa v_\varepsilon|_\mu^2 &=& \int_\theta \int_\kappa |\partial_\kappa
  v_\varepsilon|^2 \mathcal{N}\,\mathcal{M}\,d\kappa d\theta \\
  \nonumber 
  &\geq&  \int_\theta C^{-1} \int_\kappa |v_\varepsilon- \overline{v}_\varepsilon|^2
  \mathcal{N}\,d\kappa\,\mathcal{M}\,d\theta \\
  &\geq& C^{-1} |v_\varepsilon- \overline{v}_\varepsilon|_\mu^2. \label{eq:v_dk_v}
\end{eqnarray}
Since $v_\varepsilon \in L^\perp$, for all $\Phi(\theta) \in L$, we have:
\begin{displaymath}
  \int_{\theta,\kappa} v_\varepsilon(\theta,\kappa) \Phi(\theta) \mathcal{M}(\theta)
  \mathcal{N}(\kappa)\,d\theta d\kappa = \int_\theta \overline{v}_\varepsilon(\theta)
  \Phi(\theta) \mathcal{M}\,d\theta = 0.
\end{displaymath}
Therefore $\overline{v}_\varepsilon(\theta) = 0$. Combining the inequality
(\ref{eq:v_dk_v}) with (\ref{eq:d_k_v_ep}) yields:
\begin{displaymath}
  |v_\varepsilon|_\mu^2 \stackrel{\varepsilon \rightarrow 0}{\longrightarrow} 0.
\end{displaymath}
It remains to prove that $\Phi_\varepsilon$ converges to zero. With this aim, we take the
scalar product of the equation (\ref{eq:U_ep}) against the function $\kappa$. Once we take
the limit $\varepsilon\rightarrow0$, we find:
\begin{displaymath}
  <\mathcal{L}U_\varepsilon,\,\kappa>_\mu \stackrel{\varepsilon \rightarrow 0}{\longrightarrow} 0.
\end{displaymath}
Using that $|\partial_\kappa v_\varepsilon|_\mu^2$ also converges to zero, we deduce that:
\begin{equation}
  \label{eq:d_th_U}
  \int_{\theta,\kappa} \kappa^2 \partial_\theta U_\varepsilon\,\mu\,d\theta d\kappa
  \stackrel{\varepsilon \rightarrow 0}{\longrightarrow} 0.
\end{equation}
We would like to use once again a Poincar\'{e} inequality. With this aim, we define the
function $h_\varepsilon(\theta)$ as:
\begin{displaymath}
  h_\varepsilon(\theta) = \int_\kappa \kappa^2
  U_\varepsilon(\theta,\kappa)\,\mathcal{N}(\kappa)\,d\kappa
\end{displaymath}
and we use the notation:
\begin{displaymath}
  |h(\theta)|_\mathcal{M}^2 = \int_\theta |h(\theta)|^2\,\mathcal{M}\,d\theta.  
\end{displaymath}
So equation (\ref{eq:d_th_U}) can be read as $|\partial_\theta
h_\varepsilon|_\mathcal{M}^2 \stackrel{\varepsilon \rightarrow 0}{\longrightarrow} 0$.
The usual Poincar\'{e} inequality gives:
\begin{equation}
  \label{eq:Poincare2}
  |h_\varepsilon-\overline{h}_\varepsilon|_{\mathcal{M}}^2 \leq C |\partial_\theta
  h_\varepsilon|_{\mathcal{M}}^2,
\end{equation}
with $\overline{h}_\varepsilon = \int_\theta h_\varepsilon(\theta)
\mathcal{M}(\theta)\,d\theta$. But since we already know that $U_\varepsilon$
converges weakly to zero, we have:
\begin{displaymath}
  \overline{h}_\varepsilon = <U_\varepsilon,\,\kappa^2>_\mu \stackrel{\varepsilon \rightarrow 0}{\longrightarrow} 0.
\end{displaymath}
Therefore the Poincar\'{e} inequality (\ref{eq:Poincare2}) yields $h_\varepsilon
\stackrel{\varepsilon \rightarrow 0}{\longrightarrow} 0$, or in other words:
\begin{equation}
  \label{eq:cv_without_dth}
  \int_{\theta,\kappa} \kappa^2 U_\varepsilon\,\mu\,d\theta d\kappa \stackrel{\varepsilon
    \rightarrow 0}{\longrightarrow} 0.
\end{equation}
Since $v_\varepsilon$ converges to zero, equation (\ref{eq:cv_without_dth}) leads to:
\begin{displaymath}
  \int_{\theta,\kappa} \Phi_\varepsilon(\theta) \kappa^2 \mathcal{M}(\theta)
  \mathcal{N}(\kappa)\,d\theta d\kappa \stackrel{\varepsilon \rightarrow
    0}{\longrightarrow} 0,
\end{displaymath}
which finally gives that $\Phi_\varepsilon$ also converges strongly to zero in $L_\mu^2$.\\
Since both $v_\varepsilon$ and $\Phi_\varepsilon$ convergence strongly to zero,
$U_\varepsilon$ converges strongly to zero as well. This contradicts that
$|U_\varepsilon|_\mu=1$ for all $\varepsilon$. Therefore, the sequence $\psi_\varepsilon$
is bounded in $L_\mu^2$, so we can extract a subsequence which converges weakly to
$\psi_0$ in $L^2_\mu$. This function $\psi_0$ has to satisfy:
\begin{displaymath}
  \mathcal{L}\psi_0 = \chi
\end{displaymath}
which ends the proof of the lemma.

\endproof

\section*{Appendix A2:  Proof of lemma \ref{lem:sol_eq_elliptic} (ii) (probabilistic proof)}

\noindent {\bf Proof.} The operator $\mathcal{L}$ is the infinitesimal generator of the
following stochastic differential equation:
\begin{eqnarray}
  \label{eq:theta_linear}
  d\theta &=& \kappa dt , \\
  \label{eq:kappa_linear}
  d\kappa &=& -\lambda(\sin \theta + \kappa)\,dt + \sqrt{2} \alpha\, dB_t, 
\end{eqnarray}
For any function $\varphi$ regular enough, we can define the semi-group:
\begin{displaymath}
  P_t(\varphi)(\theta,\kappa) = \mathbb{E}[\varphi(X_t) | X_0=(\theta,\kappa)],
\end{displaymath}
with $X_t$ the stochastic process solution of
(\ref{eq:theta_linear})-(\ref{eq:kappa_linear}). This defines a solution of the following
equation (see \cite{oksendal_stochastic_1992}):
\begin{displaymath}
  \left\{ \begin{array}{l}
      \partial_t u = \mathcal{L}u \\
      u_{t=0} = \varphi.
    \end{array} \right.
\end{displaymath}
In particular, if we define $u(t)=P_t(\chi)$, a simple integration by part leads to:
\begin{equation}
  \label{eq:u_temp}
  u(t) - \chi = \int_0^t \mathcal{L}u(s)\,ds.
\end{equation}
Therefore, we will find a solution to (\ref{eq:elliptic_eq}) if we are able to prove that
$u(t) \stackrel{t \rightarrow \infty}{\longrightarrow} 0$. For that, we first notice that
the equilibrium measure associated with $\mathcal{L}$ is given by $\mu$ (\ref{eq:mu}) and
its adjoint operator in $L_\mu^2$ is given by:.
\begin{displaymath}
  \mathcal{L}^*\psi = -\kappa\partial_\theta \psi + \lambda\sin \theta \partial_\kappa\psi
  - \lambda\kappa\partial_\kappa\psi + \alpha^2\partial_\kappa^2\psi.
\end{displaymath}
Moreover, we can find a Lyapunov function associated with $\mathcal{L}$. The function
$V(\theta,\kappa)=1+\kappa^2$ satisfies:
\begin{eqnarray*}
  \mathcal{L}^*V &=& 2\lambda\sin \theta\, \kappa - 2\lambda\kappa^2 +2\alpha^2 \\
  &\leq& 2\lambda\kappa - \lambda(1+\kappa^2) -\lambda\kappa^2 + \lambda  + 2\alpha^2 \\
  &\leq& -\lambda V + 2(\alpha^2+\lambda)\mathds{1}_{ \{ |\kappa|\leq 2 + \sqrt{1+(2\alpha/\lambda+1)^2} \} }.
\end{eqnarray*}
Therefore $V$ is a Lyapunov function in the sense of
\cite[Def. 1.1]{bakry_rate_2007}. Since $B = \mathbb{S}^1\times\{|\kappa| \leq 2 +
\sqrt{1+(2\alpha/\lambda+1)^2}\}$ is compact, $B$ is a ``petite set'' in the terminology
\cite[Def. 1.1]{bakry_rate_2007} of Meyn \& Tweedie\cite{meyn_stability_1993}. So we can
apply \cite[Th. 2.1]{bakry_rate_2007} and conclude that there exists a constant $K_2> 0$
such that for all bounded function $\varphi$ satisfying $\int_{\theta,\kappa} \varphi\,
\mu \,d\theta d\kappa =0$, we have:
\begin{displaymath}
  |P_t(\varphi)|_\mu \leq K_2 \|\varphi\|_\infty\,\expo^{-\lambda t}.
\end{displaymath}
Therefore, we can pass to the limit $t\rightarrow\infty$ in (\ref{eq:u_temp}) to find that:
\begin{displaymath}
  -\chi = \int_0^\infty \mathcal{L}u(s)\,ds,
\end{displaymath}
Defining the function $\psi = -\int_0^\infty u(s)\,ds$, we get a solution to:
\begin{displaymath}
  \mathcal{L}\psi = \chi.
\end{displaymath}
For the uniqueness of the solution, we proceed as in appendix A1.

\endproof

\section*{Appendix B:  Proof of proposition \ref{ppo:psi_X}.}

\noindent {\bf Proof.} We first prove the following lemma.

\begin{lemma}
  \label{eq:coeff_L}
  For every integer $m$ and every positive integer $n\geq0$, we have:
  \begin{displaymath}
    \mathcal{L}(\varphi_m P_n) = \sum_{\substack{-1\leq j \leq 1\\ -1\leq k\leq1}}
    D^{m,n}(j,k) \varphi_{m+j}P_{n+k}
  \end{displaymath}
  with $D^{m,n}$ a $3\times3$ matrix given by:
  \begin{equation}
    \label{eq:D_matrix}
    D^{m,n} = \left[
      \begin{array}{ccccc}
        \ds -\beta_2\,\sqrt{n}  & & 0 & & \ds \beta_2\,\sqrt{n\!+\!1} \\
        \\
        \ds \beta_1\,m\sqrt{n} & & -\lambda n & & \ds \beta_1 \,m\sqrt{n\!+\!1} \\
        \\
        \ds \beta_2\,\sqrt{n} & & 0 & & \ds -\beta_2\,\sqrt{n\!+\!1}
      \end{array}
    \right]
  \end{equation}
  with:
  \begin{displaymath}
    \beta_1 = \frac{i\alpha}{\sqrt{\lambda}} \quad , \quad \beta_2 =
    \frac{i\lambda\sqrt{\lambda}}{4\alpha}.
  \end{displaymath}
\end{lemma}
\begin{proof}
  First, using the properties of the Hermite polynomials\footnote{Indeed $H_n'=n H_{n-1}$
    and $x H_{n} = H_{n+1} + n H_{n-1}$}, we can find several properties of $P_n$:
  \begin{eqnarray}
    \nonumber
    P_n' &=& \frac{\sqrt{\lambda}}{\alpha} \sqrt{n}\,P_{n-1} ,\\
    \label{eq:P_kappa}
    \kappa P_n &=& \frac{\alpha}{\sqrt{\lambda}}  \left(\sqrt{n+1}
      P_{n+1} + \sqrt{n} P_{n-1} \right).
  \end{eqnarray}
  In particular, the polynomials $P_n$ are eigenfunctions of the self-adjoint part of
  $\mathcal{L}$:
  \begin{equation}
    \label{eq:P_n_eigenfunction}
    - \lambda\kappa\partial_\kappa P_n + \alpha^2 \partial_k^2 P_n = -\lambda n P_n.
  \end{equation}
  Then, we compute:
  \begin{displaymath}
    \mathcal{L}(\varphi_m P_n) = \kappa P_n \partial_\theta \varphi_m -\lambda \sin
    \theta\,\varphi_m\, \partial_\kappa P_n \,+\,\varphi_m(- \lambda\kappa\partial_\kappa
    P_n + \alpha^2 \partial_k^2 P_n).
  \end{displaymath}
  The derivative of $\varphi_m$ with respect to $\theta$ is given by:
  \begin{eqnarray*}
    \partial_\theta \varphi_m &=& \partial_\theta
    \left(\frac{\expo^{im\theta}}{\sqrt{2\pi\mathcal{M}}}\right) = im
    \left(\frac{\expo^{im\theta}}{\sqrt{2\pi\mathcal{M}}}\right) +
    \frac{\expo^{im\theta}}{\sqrt{2\pi}}\left(-\frac{1}{2} \frac{-\frac{\lambda^2}{\alpha^2} \sin \theta
        \mathcal{M}}{\mathcal{M}^{3/2}}\right)\\
    &=& im \varphi_m + \frac{\lambda^2}{2 \alpha^2}\frac{\expo^{im\theta}}{\sqrt{2\pi}}
    \left(\frac{\expo^{i\theta}-\expo^{-i\theta}}{2i}\,\frac{1}{\sqrt{\mathcal{M}}}\right) \\
    &=& im \varphi_m - \frac{i\lambda^2}{4 \alpha^2}(\varphi_{m+1}-\varphi_{m-1}).
  \end{eqnarray*}
  Using (\ref{eq:P_kappa}), we also have:
  \begin{eqnarray}
    \nonumber
    \hspace{-.2cm} \kappa P_n \partial_\theta \varphi_m &=& \frac{\alpha}{\sqrt{\lambda}}
    \left(\sqrt{n\!+\!1} P_{n+1} \!+\! \sqrt{n} P_{n-1}
    \right)\,\left(im \varphi_m - \frac{i\lambda^2}{4 \alpha^2}(\varphi_{m+1}-\varphi_{m-1})\right).\\
    &=& \frac{i\alpha}{\sqrt{\lambda}} \left(m \sqrt{n\!+\!1} P_{n+1} \varphi_m +  m \sqrt{n}
      P_{n-1}\varphi_m \right) \label{eq:part1}\\
    &&-\frac{i\lambda\sqrt{\lambda}}{4 \alpha} \left(\sqrt{n\!+\!1} P_{n+1} \varphi_{m+1} +
      \sqrt{n} P_{n-1} \varphi_{m+1}  \right) \label{eq:part2}\\
    && +\frac{i\lambda\sqrt{\lambda}}{4 \alpha} \left(\sqrt{n\!+\!1} P_{n+1} \varphi_{m-1}  + 
      \sqrt{n} P_{n-1} \varphi_{m-1} \right). \label{eq:part3}
  \end{eqnarray}
  Thus, we have:
  \begin{eqnarray}
    \nonumber 
    -\lambda \sin \theta\,\varphi_m\, \partial_\kappa P_n &=& -\lambda
    \frac{\expo^{i\theta}-\expo^{-i\theta}}{2i}\,\varphi_m\,\frac{\sqrt{\lambda}}{\alpha}
    \sqrt{n}\,P_{n-1} \\
    &=& \frac{i\lambda\sqrt{\lambda}}{2\alpha}\sqrt{n} (\varphi_{m+1}-\varphi_{m-1})P_{n-1}. \label{eq:part4}
  \end{eqnarray}
  Finally, since $P_n$ satisfies (\ref{eq:P_n_eigenfunction}), we get:
  \begin{equation}
    \varphi_m(- \lambda\kappa\partial_\kappa P_n + \alpha^2 \partial_k^2 P_n) = -\lambda n\,\varphi_m P_n. \label{eq:part5}
  \end{equation}
  Combining
  (\ref{eq:part1})~(\ref{eq:part2})~(\ref{eq:part3})~(\ref{eq:part4})~(\ref{eq:part5}), we
  find the expression (\ref{eq:D_matrix}) of $D^{m,n}$.
\end{proof}
To find the matrix representation of the operator $\mathcal{L}$ in $V_{m,n}$, we introduce
the vectors $u$ and $v$ defined by:
\begin{eqnarray*}
  u &=& (\varphi_{-m},\,\dots,\,\varphi_0,\,\dots,\,\varphi_m)^T \\
  v &=& (P_0,\,\dots,\,P_n)^T.
\end{eqnarray*}
With these notations, a function $\psi\in V_{m,n}$ with a matrix representation $X$
(\ref{eq:matrix_X}) can be written as:
\begin{displaymath}
  \psi_{m,n}  = \sum_{|j|\leq m,\,0\leq k\leq n} \!\!\!\!C_j^k \varphi_j P_k \,=\, u^T\,X\,v.
\end{displaymath}
Moreover, thanks to the matrices defined in (\ref{eq:L_mp}), we can write for example
\begin{eqnarray*}
  u^T\,X L_{-1}\, v &=& \sum_{|j|\leq m,\,1\leq k\leq n} C_j^k \varphi_j P_{k-1} \\
  (D_1 u)^T\,X\,v &=& \sum_{|j|\leq m,\,1\leq k\leq n} j C_j^k \varphi_j P_k.
\end{eqnarray*}
For a function $\psi\in L_\mu^2$, using the lemma \ref{eq:coeff_L}, we can write:
\begin{eqnarray*}
  \mathcal{L}\psi &=& \sum_{m,n} C_m^n \left(\sum_{\substack{-1\leq j \leq 1\\ -1\leq
        k\leq 1}} D^{m,n}(j,k)
    \varphi_{m+j}P_{n+k} \right) \\
  &=& \sum_{m,n} C_m^n \left(\beta_1m\varphi_{m} (\sqrt{n}P_{n-1} + \sqrt{n+1}P_{n+1}) \right.\\
  && \hspace{1cm}\ds + \beta_2 (-\varphi_{m-1}+\varphi_{m+1})\sqrt{n}P_{n-1}\\
  && \hspace{1.5cm}\ds+ \beta_2 (\varphi_{m-1}-\varphi_{m+1})\sqrt{n+1}P_{n+1} \;- \lambda
  \varphi_m\,n P_n \big).
\end{eqnarray*}
Therefore, for every $\varphi\in V_{m,n}$, we have:
\begin{eqnarray*}
  < \mathcal{L}\psi , \, \varphi>_\mu &=& <\;\beta_1 (D_1 u)^T X (\sqrt{D_2}L_{-1} v) + \beta_1 (D_1
  u)^T X (L_{+1}\sqrt{D_2} v) \\
  && - \beta_2 (L_{-1} u)^T X (\sqrt{D_2}L_{-1} v) + \beta_2 (L_{+1} u)^T X (\sqrt{D_2}L_{-1} v) \\
  && + \beta_2 (L_{-1} u)^T X (L_{+1}\sqrt{D_2}v) - \beta_2 (L_{+1} u)^T X (L_{+1}\sqrt{D_2}v) \\
  && - \lambda u^T X (D_2 v) \quad, \quad \varphi>_\mu.
\end{eqnarray*}
We can simplify this expression:
\begin{eqnarray*}
  < \mathcal{L}\psi , \, \varphi>_\mu &=& <\; u^T \,\beta_1 D_1 X (\sqrt{D_2}L_{-1}   +  L_{+1}\sqrt{D_2})\,v \\
  && + u^T\,\beta_2 (-L_{-1} + L_{+1})^T X (\sqrt{D_2}L_{-1} )\,v \\
  && + u^T\,\beta_2 (L_{-1} - L_{+1})^T X (L_{+1} \sqrt{D_2})\,v \\
  && - \lambda u^T X (D_2 v) \quad, \quad \varphi>_\mu ,
\end{eqnarray*}
which finally gives
\begin{displaymath}
  < \mathcal{L}\psi , \, \varphi>_\mu = <\;\beta_1 M_1 X N_1 \,+\, \beta_2 M_2 XN_2 \,-\,
  \lambda X D_2 \quad, \quad \varphi>_\mu ,
\end{displaymath}
with $M_1,\,M_2,\,N_1$ and $N_2$ defined in (\ref{eq:MN}). Therefore, using $\varphi=-\sin
\theta$, we find that $X$ has to satisfy equation (\ref{eq:X_equation}).

\endproof


\end{document}